\documentclass[11pt,reqno]{amsart}

\usepackage[colorlinks=true, pdfstartview=FitV, linkcolor=blue,
citecolor=blue]{hyperref}

\usepackage{a4wide}
\usepackage{amsmath,amssymb}
\usepackage{bbm}
\usepackage{xcolor}
\usepackage{latexsym}
\usepackage{tikz-cd}

\allowdisplaybreaks

\usepackage{scalerel,stackengine}

\stackMath
\newcommand\reallywidehat[1]{%
\savestack{\tmpbox}{\stretchto{%
  \scaleto{%
    \scalerel*[\widthof{\ensuremath{#1}}]{\kern-.6pt\bigwedge\kern-.6pt}%
    {\rule[-\textheight/2]{1ex}{\textheight}}
  }{\textheight}%
}{0.5ex}}%
\stackon[1pt]{#1}{\tmpbox}%
}

\newcommand\reallywidecheck[1]{%
\savestack{\tmpbox}{\stretchto{%
  \scaleto{%
    \scalerel*[\widthof{\ensuremath{#1}}]{\kern-.6pt\bigwedge\kern-.6pt}%
    {\rule[-\textheight/2]{1ex}{\textheight}}
  }{\textheight}%
}{0.5ex}}%
\stackon[1pt]{#1}{\scalebox{-1}{\tmpbox}}%
}

\numberwithin{equation}{section}

\newcommand{\Z}{{\mathbb Z}}

\newcommand{\R}{{\mathbb R}}
\newcommand{\N}{{\mathbb N}}
\newcommand{\MM}{{\mathbb M}}
\newcommand{\CC}{{\mathbb C}}
\newcommand{\KK}{{\mathbb K}}

\newcommand{\XX}{{\mathbb X}}

\newcommand{\cA}{{\mathcal A}}

\newcommand{\im}{{\mathrm{i}}}
\newcommand{\supp}{{\operatorname{supp}}}

\newcommand{\e}{\operatorname{e}}

\newcommand{\dd}{\mbox{d}}

\newcommand{\eps}{\varepsilon}
\newcommand{\cM}{{\mathcal M}}
\newcommand{\cL}{{\mathcal L}}

\newcommand{\cS}{{\mathcal S}}

\newcommand{\lm}{\ensuremath{\lambda\!\!\!\lambda}}

\newcommand{\Cu}{C_{\mathsf{u}}}
\newcommand{\Cc}{C_{\mathsf{c}}}
\newcommand{\Cz}{C^{}_{0}}

\newcommand{\limn}{\lim_{n\to\infty}}

\theoremstyle{plain}
\newtheorem{theorem}{Theorem}[section]
\newtheorem{proposition}[theorem]{Proposition}
\newtheorem{lemma}[theorem]{Lemma}
\newtheorem{coro}[theorem]{Corollary}

\theoremstyle{definition}
\newtheorem{definition}[theorem]{Definition}
\newtheorem{remark}[theorem]{Remark}
\newtheorem{example}[theorem]{Example}
\newtheorem{question}[theorem]{Question}

\begin{document}
\title{On the (dis)continuity of the Fourier transform of measures}

\author{Timo Spindeler}
\address{Department of Mathematical and Statistical Sciences, \newline
\hspace*{\parindent}632 CAB, University of Alberta, Edmonton, AB, T6G 2G1, Canada}
\email{spindele@ualberta.ca}

\author{Nicolae Strungaru}
\address{Department of Mathematical Sciences, MacEwan University \newline
\hspace*{\parindent} 10700 -- 104 Avenue, Edmonton, AB, T5J 4S2, Canada\\
and \\
Institute of Mathematics ``Simon Stoilow''\newline
\hspace*{\parindent}Bucharest, Romania}
\email{strungarun@macewan.ca}
\urladdr{http://academic.macewan.ca/strungarun/}

\begin{abstract}
In this paper, we will study the continuity of the Fourier transform of measures with respect to the vague topology. We show that the Fourier transform is vaguely discontinuous on $\R$, but becomes continuous when restricting to a class of Fourier transformable measures such that either the measures, or their Fourier transforms are equi-translation bounded. We discuss continuity of the Fourier transform in the product and norm topology. We show that vague convergence of positive definite measures implies the equi translation boundedness of the Fourier transforms, which explains the continuity of the Fourier transform on the cone of positive definite measures. In the appendix, we characterize vague precompactness of a set a measures in arbitrary LCAG, and the necessity of second countability property of a group for defining the autocorrelation measure. 
\end{abstract}



\maketitle

\section{Introduction}

The Fourier transform of measures on locally compact Abelian groups (LCAG) and its continuity play a central role for the theory of mathematical diffraction. When restricting to the cones of positive and positive definite measures, the Fourier transform is continuous with respect to the vague topology \cite{BF}. In fact, as shown in \cite{MoSt}, only positive definitedness is important: the Fourier transform is vaguely continuous from the cone of positive definite measures on $G$ to the cone of positive measures on the dual group $\widehat{G}$.

Introduced by Hof \cite{Hof,Hof3}, mathematical diffraction is defined as follows. Given a translation bounded measure $\mu$ which models a solid, the autocorrelation measure $\gamma$ is defined as the vague limit of the 2-point correlations $\gamma_n$ of finite sample approximations $\mu_n$ of $\mu$. The measure $\gamma$ is positive definite, and hence Fourier transformable, and its Fourier transform $\widehat{\gamma}$ models the diffraction of $\mu$. The continuity of the Fourier transform on the cone of positive definite measures ensures that the diffraction measure $\widehat{\gamma}$ is the vague limit of the diffractions $\widehat{\gamma_n}$ of the finite sample approximations $\mu_n$.

If $G=\R^d$, as $\widehat{\gamma}$ is translation bounded \cite{ARMA1}, it is also the distributional Fourier transform of $\gamma$ \cite{NS12}. While the Fourier theory of tempered distributions is more established than the Fourier theory of measures, in many situations it is more convenient to work with measures. Given a regular model set $\Lambda \subseteq \R^d$, it was shown in \cite{CRS,CRS2} that its autocorrelation $\gamma$, diffraction $\widehat{\gamma}$ and the pure point nature of $\Lambda$ can be deduced from the Fourier analysis of the lattice $\cL$ in the underlying cut and project scheme (CPS) $(\R^d, H, \cL)$, which is simplified by the Poisson Summation Formula (PSF). In many situations, the group $H$ is a compactly generated LCAG \cite{NS11}, which may be non-Euclidian. This happens often for equal length substitutions \cite{TAO}. Some of these ideas can be used for the larger class of weak model sets of extremal density \cite{BHS,KR}.

Going beyond $\R^d$, one could work with the Bruhat--Schwartz theory instead of measures (see \cite{Osb} for example), but this would lead to many inconveniences. The mathematical theory for diffraction is well established using the Fourier analysis for measures, but nothing is done yet for the Bruhat--Schwartz theory, and it is not the goal of this paper to do this. The deep, yet subtle connection between pure point Fourier transform/Lebesgue decomposition and strong almost periodicity/Eberlein decomposition has important consequences for the diffraction theory (see \cite{TAO,TAO2,BL,BL2,BM,LS,SpSt,NS11,NS12,NS13} just to name a few), and is very well understood for translation bounded measures \cite{ARMA,MoSt}, and was recently extended to tempered distribution in $\R^d$ \cite{ST}, but again nothing was done yet for the Bruhat--Schwartz space. Because of these reasons, we will work with the Fourier analysis of measures.

\smallskip

As we mentioned above the Fourier transform is continuous on the cone of positive definite measures on $G$. Some of the recent developments in the theory of long range order shows the need of going outside this cone. We list two such directions below.

In \cite{NS13}, the author proved the existence of the generalized Eberlein decomposition and the almost periodicity of each diffraction spectral component for measures with Meyer set support. The key ingredient is the ping-pong lemma for Meyer sets \cite[Lem.~3.3]{NS13}, which says that there exists a class of twice Fourier transformable measures with pure Fourier transform and  supported inside model sets, with the property that a Fourier transformable measure $\gamma$ has Meyer set support if and only if there exists a measure $\omega$ in this class and a finite measure $\nu$ such that
\begin{displaymath}
\widehat{\gamma}=\widehat{\omega}*\nu \,.
\end{displaymath}
This allows one to go back and forth between $G$ and the Fourier dual group $\widehat{G}$, and use the long-range order properties of $\omega$ and $\widehat{\omega}$, which are inherited from the lattice $\cL$ in the underlying CPS \cite{CRS,CRS2}, to deduce the long-range order of Meyer sets. Typically, the measure $\omega$ used in the ping-pong lemma is not positive definite. Indeed, a trivial application of the Krein's inequality shows that whenever $\omega$ is positive definite, the original measure $\gamma$ must be supported inside a fully periodic set. It is worth noting that the measure $\omega$ only depends on a covering model set for $\supp(\gamma)$, and given a family $(\gamma_\alpha)_\alpha$ of measures supported inside a common model set, one can study the family $( \widehat{\gamma_\alpha})_{\alpha}$ of their Fourier transforms by using the same measure $\omega$. In this case, one needs to deal with Fourier transforms of non-positive definite measures, and this is a direction we may investigate in the future.

Recently, there was much progress done in the investigation of 1-dimensional substitution tilings done via the renormalisation equations (see for example \cite{BaGa,BG,BG2,BG3,BGM,BGM} just to name a few).
The renormalisation equations require one to work with the Eberlein convolutions $\gamma_{ij}:=\delta_{\Lambda_i} \circledast \widetilde{\delta_{\Lambda_j}}$ of the typed end points of the substitutions. While these measures are Fourier transformable, they are not positive definite. As these play a central role in some of the new developments, it becomes important to understand the properties Fourier transform outside the cone of positive definite measures.

\smallskip

A standard folklore result is the continuity of the Fourier transform with respect to the vague topology, i.e. if $(\mu_\alpha)_{\alpha}$ is a net of measures that converges vaguely to some measure $\mu$, then the net $(\widehat{\mu_\alpha})_{\alpha}$ converges to $\widehat{\mu}$ in the vague topology. As we mentioned above, this is proven for positive definite measures (see Proposition~\ref{prop:pd_vauge} below). We will prove in this paper that the result is false in general, but holds under some extra restrictions, namely if either the measures are equi-translation bounded or if their Fourier transforms are equi-translation bounded. In order to understand the subtle issues we are dealing with, we will restrict in this paper to the case $G=\R^d$, which allows us to use the theory of distributions. Outside the class of translation bounded measures the problems seem to be much more complicated and we find more questions than answers.

\smallskip

The paper is organized as follows. In Section~\ref{Sect 2}, we characterize the vague compactness of a set of measures and use this to discuss the vague convergence of a sequence of measures in $\R^d$. In Section~\ref{sec:vague}, we discuss the (dis)continuity of the Fourier transform in the vague topology. We provide two examples (Example \ref{EX2} and Example \ref{ex3}) of vague null sequences for which the Fourier transforms are not vague null, and discuss the connection among the vague convergence, convergence in the distribution topology and convergence in the tempered distribution topology. In particular, we show that the three topologies coincide on any set which is equi-translation bounded, and we provide an example (Example \ref{ex:tempdis}) of a sequence of measures which is null in the (tempered) distribution topology but not vaguely null. Then, we proceed to prove two of the main results in the paper, Theorem~\ref{T3} and Theorem~\ref{thm:mainvag2}. In Theorem~\ref{T3}, we show that if a sequence $\mu_n$ of Fourier transformable measures converges vaguely to a tempered measure $\mu$, and if $\{ \widehat{\mu_n}\, :\,  n\in\N\}$ is equi translation bounded, then $\mu$ is Fourier transformable and $(\widehat{\mu_n})_{n\in\N}$ converges vaguely to $\mu$. In Theorem~\ref{thm:mainvag2}, we prove a complementary result: given a sequence $(\mu_n)_{n\in\N}$ of equi translation bounded measures, which converges vaguely to some measure $\mu$, then $\mu$ is Fourier transformable and $(\widehat{\mu_n})_{n\in\N}$ converges vaguely to $\widehat{\mu}$ if and only if $\{ \widehat{\mu_n}\, :\, n\in\N \}$ has compact vague closure. We complete the section by looking in Example \ref{ex417} at a sequence of finite measures which is not convergent, but for which the Fourier transforms are vaguely null.

In Sections \ref{sec:product} and \ref{sec:norm}, we look at necessary conditions for the convergence of the Fourier transform in the product topology and norm topology, respectively. As an application, in Section~\ref{sec:application}, we give a necessary and sufficient condition for a measure to be Fourier transformable.

In Section \ref{sec:pd}, we explain the continuity of the Fourier transform on the cone of positive definite measures. We show that on this cone, the Fourier transform is vague to norm bounded. As this is a cone, boundedness does not imply continuity, but implies that, whenever a sequence (net) of positive definite measures is vaguely convergent, their Fourier transforms are equi-translation bounded, and the continuity of the Fourier transform follows immediately from Theorem~\ref{T3}.

Finally, as some of the results in Section~\ref{Sect 2} are of interest in general, we prove them for arbitrary LCAG in Appendix~\ref{sect:app}. We also show that the hull $\XX(\mu)$  of a measure is compact if and only if $\mu$ is translation bounded, generalizing a result from \cite{BL}.

\section{Preliminaries}

First, let us introduce some terms and concepts which we will need in the following sections.
We use the familiar symbols $C_{\text{c}}(\R^d)$ and $C_{\text{u}}(\R^d)$ for the spaces of compactly supported continuous and bounded uniformly continuous functions, respectively, which map from $\R^d$ to $\CC$. For any function $g$ on $\R^d$, the functions $T_tg$ and $g^{\dagger}$ are defined by
\begin{displaymath}
(T_tg)(x):=g(x-t)\quad \text{ and } \quad g^{\dagger}(x):=g(-x).
\end{displaymath}

A \emph{measure} $\mu$ on $\R^d$ is a linear functional on $C_{\text{c}}(\R^d)$ such that, for every compact subset $K\subseteq \R^d$, there is a constant $a_K>0$ with
\begin{displaymath}
|\mu(g)| \leqslant a_{K}\, \|g\|_{\infty}
\end{displaymath}
for all $g\in C_{\text{c}}(\R^d)$ with $\supp(g) \subseteq K$. Here, $\|g\|_{\infty}$ denotes the supremum norm of $g$. By the Riesz Representation theorem, this definition is equivalent to the classical measure theory concept of regular Radon measure.

For a measure $\mu$ on $\R^d$, we define  $T_t\mu$ and $\mu^{\dagger}$ by
\begin{displaymath}
(T_t\mu)(g):= \mu(T_{-t}g)\quad  \text{ and } \quad
\mu^{\dagger}(g):= \mu(g^{\dagger}).
\end{displaymath}

Given a measure $\mu$, there exists a positive measure $| \mu|$ such that, for all $f \in \Cc(\R^d)$ with $f \geqslant 0$, we have \cite{Ped} (compare \cite[Appendix]{CRS3})
\[
| \mu| (f)= \sup \{ \left| \mu (g) \right| \ :\ g \in \Cc(\R^d),\,  |g| \leqslant f \} \,.
\]
The measure $| \mu|$ is called the \emph{total variation of} $\mu$.

\begin{definition}
A measure $\mu$ on $\R^d$ is called \emph{Fourier transformable} if there exists a measure $\widehat{\mu}$ on $\R^d$ such that
\[
\reallywidecheck{f}\in L^2(\widehat{\mu}) \qquad \text{ and } \qquad
\left\langle \mu\, , \, f*\widetilde{f} \right\rangle =
\left\langle \widehat{\mu}\, , \, |\reallywidecheck{f}|^2 \right\rangle
\]
for all $f\in\Cc(\R^d)$. In this case, $\widehat{\mu}$ is called the \emph{Fourier transform} of $\mu$.
We will denote the space of Fourier transformable measures by $\cM_T(\R^d)$.
\end{definition}

\begin{remark}
It was shown in \cite{CRS} that a measure $\mu$ on $\R^d$ is Fourier transformable if and only if there is a measure $\widehat{\mu}$ on $\R^d$ such that
\[
\reallywidecheck{f}\in L^1(\widehat{\mu}) \qquad \text{ and } \qquad
\left\langle \mu\, , \, f \right\rangle =
\left\langle \widehat{\mu}\, , \reallywidecheck{f} \right\rangle
\]
for all $f\in KL(\R^d):=\{g\in\Cc(\R^d)\ :\ \widehat{g}\in L^1(\R^d)\}$.
\end{remark}

We will often use the following result.

\begin{theorem}\cite[Thm.~7.2]{ARMA1} Let $\mu \in \cM_T(\R^d)$. Then, $\mu$ is a tempered measure.
\end{theorem}

Whenever we deal with Fourier transformable measures, this result will allow us consider them as tempered distributions.

\smallskip

The next property will turn out to be quite useful, when we want to give sufficient conditions for the continuity of the Fourier transform.

\begin{definition}
A measure $\mu$ on $\R^d$ is called \emph{translation bounded} if
\[
\|\mu\|_K:=\sup_{t\in \R^d} |\mu|(t+K)<\infty \,,
\]
for all compact sets $K\subseteq \R^d$.

As usual, we will denote by $\cM^\infty(G)$ the space of translation bounded measures. We will use the notation
\begin{displaymath}
\cM^\infty_{T}(\R^d):=\cM^\infty(\R^d) \cap \cM_{T}(\R^d) \,.
\end{displaymath}

A family of measures $(\mu_{\alpha})_{\alpha}$ is called \emph{equi translation bounded} if
\[
\sup_{\alpha} \|\mu_{\alpha}\|_K < \infty \,,
\]
for all compact sets $K\subseteq \R^d$.
\end{definition}

In fact, it suffices to show that $\|\mu\|_K:=\sup_{t\in \R^d} |\mu|(t+K)<\infty$ holds for a single compact set $K$. Furthermore, one can show that $\mu$ is translation bounded if and only if $\mu*f\in \Cu(\R^d)$, for all $f\in\Cc(\R^d)$.

\bigskip
Next, let us introduce the different kind of convergences we are going to work with.

\begin{definition}
Let $(\mu_n)_{n\in\N}$ be a sequence of measures on $\R^d$, and let $\mu\in\cM(\R^d)$. Then, the sequence $(\mu_n)_{n\in\N}$ converges to $\mu$
\begin{enumerate}
\item[$\bullet$] in the \emph{vague topology} if $\lim_{n\to\infty} \mu_n(f)=\mu(f)$ for all $f\in \Cc(\R^d)$;
\item[$\bullet$] in the \emph{norm topology} if $\lim_{n\to\infty} \|\mu_n-\mu\|_K=0$ for some (fixed) non-empty and compact set $K\subseteq \R^d$ which is the closure of its interior;
\item[$\bullet$] in the \emph{product topology} if $\lim_{n\to\infty} \|(\mu_n-\mu)*g\|_{\infty}=0$ for all $g\in \Cc(\R^d)$.
\end{enumerate}
\end{definition}

\begin{definition}
Let $(\mu_n)_{n\in\N}$ be a sequence of measures on $\R^d$, and let $\mu\in\cM(\R^d)$. Then, the sequence $(\mu_n)_{n\in\N}$ converges to $\mu$
\begin{enumerate}
\item[$\bullet$] in the \emph{tempered distribution topology} if $\mu_n, \mu$ are tempered measures, for all $n\in\N$, and $\lim_{n\to\infty} \mu_n(f)=\mu(f)$ for all $f\in \cS(\R^d)$;
\item[$\bullet$] in the \emph{distribution topology} if  $\lim_{n\to\infty} \mu_n(f)=\mu(f)$ for all $f\in \Cc^\infty(\R^d)$.
\end{enumerate}
\end{definition}

\begin{remark}
Vague convergence implies convergence in the distribution topology. Also,
convergence in the tempered distribution topology implies convergence in the distribution topology.
\end{remark}

\section{A few notes on the vague topology} \label{Sect 2}

In this section we review some results about the vague theory for measures, and generalize few results of \cite{BL}. Since the proofs do not rely on the geometry of $\R^d$ and will be of interest in general, we will prove them in the general setting of LCAG $G$ in the Appendix and refer to it for the proofs.

Let us start with the following definition.

\begin{definition} A set $\mathcal{A} \subseteq \cM(\R^d)$ is called \emph{vaguely bounded} if, for each $f \in \Cc(\R^d)$, the set $\{ \mu(f): \mu \in \mathcal{A} \}$ is bounded.
\end{definition}

\begin{proposition}\label{P2}
Let $\mathcal{A} \subseteq \cM(\R^d)$. Then, the following statements are equivalent:
\begin{itemize}
\item [(i)] $\mathcal{A}$ is vaguely precompact.
\item [(ii)] $\mathcal{A}$ is vaguely bounded.
\item [(iii)] For each compact set $K \subseteq \R^d$, the set $\{ \left| \mu \right| (K) : \mu  \in \mathcal{A} \}$ is bounded.
\item[(iv)] There exists a collection $\{ K_\alpha\, :\,  \alpha\}$ of compact sets in $\R^d$ such that
\begin{itemize}
\item[$\bullet$] $\R^d =\bigcup_{\alpha} (K_\alpha)^\circ$ and
\item[$\bullet$] $\{ |\mu| (K_\alpha) : \mu  \in \mathcal{A} \}$ is bounded, for each $\alpha$.
\end{itemize}
\end{itemize}
Moreover, in this case the vague topology on $\cA$ is metrisable.
\end{proposition}
\begin{proof}
See Proposition~\ref{P2A}.
\end{proof}

Let us note the following consequence.

\begin{proposition}\label{P3}
et $(U_n)_{n\in\N}$ be a sequence of open precompact sets such that $\R^d= \bigcup_n U_n$. Let $C_n >0$ be constants, for all $n\in\N$. Then, the space
\begin{displaymath}
\MM:= \{ \mu \in \cM(\R^d)\, :\, |\mu| (U_n) \leqslant C_n \text{ for all } n\in\N \}
\end{displaymath}
is vaguely compact and metrisable.
\end{proposition}
\begin{proof}
See Proposition~\ref{P3A}.
\end{proof}

\begin{remark} It follows from Proposition~\ref{P2} that a set $\mathcal{A}$ is vaguely precompact if and only if it is a subset of some $\MM$ as in Proposition~\ref{P3}.
\end{remark}

\smallskip
We will often make use of the following result.

\begin{proposition}
Let $\mu_n, \mu$ be measures on $\R^d$, for all $n\in\N$. Then, $(\mu_n)_{n\in\N}$ converges vaguely to $\mu$ if and only if
\begin{itemize}
\item [(i)] $\{ \mu_n\, :\, n\in\N \}$ is vaguely bounded,
\item [(ii)] the set
\begin{displaymath}
D:=\{ f \in \Cc(\R^d)\, :\, \mu(f) = \lim_{n\to\infty} \mu_n (f) \}
\end{displaymath}
is dense in $\Cc(\R^d)$ with respect to the inductive topology.
\end{itemize}
\end{proposition}
\begin{proof}
$\Longrightarrow$ Is obvious.

\medskip

\noindent $\Longleftarrow$ Let $f \in \Cc(\R^d)$ be arbitrary, and let $\eps >0$.
By the density of $D$ and the definition of the inductive topology, there exist $g_m \in \Cc(\R^d)$ and a compact set $K$ such that $\supp(g_m), \supp(f) \subseteq K$ and $\lim_{m\to\infty} \| f-g_m \|_\infty =0$.

Now, by (i) and Proposition~\ref{P2}, there exists a $c>0$ such that $|\mu_n|(K),|\mu|(K) \leqslant c$ for all $n \in \N$.
Then, there exists some $m$ such that
\begin{displaymath}
\| f-g_m \|_\infty < \frac{\eps}{3 c +1} \,.
\end{displaymath}
Fix this $m$, and denote for simplicity $g:= g_m$. By the definition of $D$, there exists some $N\in\N$ such that, for all $n >N$, we have
\begin{displaymath}
\left| \mu_n(g)- \mu(g) \right| <\frac{\eps}{3} \,.
\end{displaymath}
Finally, for all $n >N$, we have
\begin{align*}
|\mu_n(f)-\mu(f)|
    &\leqslant |\mu_n(f)-\mu_n(g)| + |\mu_n(g)-\mu(g)| + |\mu(g)-\mu(f)| \\
    &< \|f-g\|_{\infty}\, |\mu_n|(K) + \frac{\eps}{3} + \|f-g\|_{\infty}\,
      |\mu|(K) \\
    &< \frac{\eps}{3} +\frac{\eps}{3} +\frac{\eps}{3}  = \eps \,.
\end{align*}
\end{proof}

\begin{coro} \label{coro:chara_vc}
Let $\mu_n, \mu$ be measures on $\R^d$, for all $n\in\N$. Then, the sequence $(\mu_n)_{n\in\N}$ converges vaguely to $\mu$ if and only if $\{ \mu_n \, :\, n\in\N \} $ is vaguely bounded and $(\mu_n)_{n\in\N}$ converges to $\mu$ in the distribution topology.
\end{coro}

\begin{coro}
Let $\mu_n, \mu$ be tempered measures on $\R^d$, for all $n\in\N$. If $\{ \mu_n\, :\, n\in\N \}$ is vaguely bounded and $(\mu_n)_{n\in\N}$ converges to $\mu$ in the tempered distribution topology, then $(\mu_n)_{n\in\N}$ converges vaguely to $\mu$.
\end{coro}

\section{Continuity in the vague topology} \label{sec:vague}

In the following sections, given a sequence $(\mu_n)_{n\in\N}$ of Fourier transformable measures, we want to study the convergence of $(\widehat{\mu_n})_{n\in\N}$ in different topologies. We will start with the vague topology.

If $(\mu_n)_{n\in\N}$ is a sequence of Fourier transformable measures which converges vaguely to a measure $\mu$, then $\mu$ is not necessarily Fourier transformable. For example, pick any measure $\mu$ on $\R$ which is not Fourier transformable. For every $n\in\N$, define $\mu_n:= \mu|_{[-n,n]}$. Then, $\mu_n$ is a finite measure for all $n\in\N$ (thus Fourier transformable), and $(\mu_n)_{n\in\N}$ converges to $\mu$ in the vague topology. Still, even if $\mu$ is Fourier transformable, the sequence $(\widehat{\mu_n})_{n\in\N}$ does not necessarily converge to $\widehat{\mu}$ in the vague topology as the next example shows.

\begin{example}\label{EX2}
For all $n\in\N$, consider the measure $\mu_n:= (1_{\R\setminus [-n,n]})\,\lm$, i.e.
\[
\mu_n(f)= \int_{-\infty}^{-n} f(x)\ \dd x + \int_n^\infty f(x)\ \dd x
\]
for all $f\in \Cc(\R^d)$. It is easy to see that each $\mu_n$ is a tempered measure. Moreover, consider the function $g(x):= (1_{[-1,1]}*1_{[-1,1]})(x)$. Since $\widehat{1_{[-1,1]}}(x) = 2\, \operatorname{sinc}(2\pi x)$, we have
\[
\widehat{\mu_n}(g) = \mu_n(\widehat{g}) = \int_{-\infty}^{-n} 4\, \operatorname{sinc}^2(2\pi x)\ \dd x +  \int_n^{\infty} 4\, \operatorname{sinc}^2(2\pi x)\ \dd x >0 \,.
\]
Next, define $\alpha_n:=\widehat{\mu_n}(g)$ and
\[
\nu_n := \frac{1}{\alpha_n}\, \mu_n \,,
\]
for all $n\in\N$. We will now show that the sequence $(\nu_n)_{n\in\N}$ has the following two properties:
\begin{enumerate}
\item[(i)] $\nu_n \to 0$ in the vague topology.
\item[(ii)] $\widehat{\nu_n} \not\rightarrow \widehat{0} =0$ in the vague topology.
\end{enumerate}
(i): This follows immediately from the fact that $\supp(\nu_n)=\R\setminus[-n,n]$.

\smallskip

\noindent (ii): First note that $g\in\Cc(\R)$. To be more precise, a simple computation shows that $g(x)= (2-|x|)\, 1_{[-2,2]}(x)$. Now, because of
\[
\widehat{\nu_n}(g) =  \frac{1}{\alpha_n}\, \widehat{\mu_n}(g) = 1 \qquad \text{ for all } n\in\N,
\]
the sequence $(\widehat{\nu_n})_{n\in\N}$ cannot converge to $0$.
\end{example}

 Here is a more explicit example.

\begin{example}\label{ex3}
Let $\mu_n= n^3 \delta_{n \Z}- n^3 \delta_0$, for all $n\in\N$. It is easy to see that $(\mu_n)_{n\in\N}$ converges to $0$ in the vague topology.

A simple computation shows that
\begin{displaymath}
\widehat{\mu_n}= n^2\delta_{\frac{1}{n} \Z}- n^3 \lambda \,.
\end{displaymath}
Let
\begin{displaymath}
f(x)= \left\{
\begin{array}{lc}
x^2 & \mbox{ if } 0 \leqslant x \leqslant 1, \\
2-x & \mbox{ if } 1 < x \leqslant 2, \\
0 &\mbox{ otherwise}.
\end{array}
\right.
\end{displaymath}
Then,
\begin{align*}
\widehat{\mu_n}(f)
  &= n^2\left( \sum_{k=1}^{2n} f\Big(\frac{k}{n}\Big) \right) - n^3\int_0^2 f(x)\ \dd x \\
  &= n^2\left( \sum_{k=1}^{n}\frac{k^2}{n^2} \right)+ n^2\left( \sum_{k=n+1}^{2n} \Big(2-\frac{k}{n}\Big) \right) - n^3 \Big(\frac{1}{3}+\frac{1}{2}\Big) \\
  &= \left( \frac{n(n+1)(2n+1)}{6} \right)+ n\left( \sum_{k=n+1}^{2n}2n-k \right) - \frac{5n^3}{6}\\
  &= \left( \frac{n(n+1)(2n+1)}{6} \right)+ n\left( \sum_{j=0}^{n-1} j \right) - \frac{5n^3}{6} \\
  &=  \frac{n(n+1)(2n+1)}{6}  +  \left( \frac{n^2(n-1)}{2} \right) -\frac{5n^3}{6} \\
  &=  \frac{n(n+1)(2n+1)+ 3n^2(n-1)-5n^3}{6} =  \frac{n}{6} \,,
\end{align*}
for all $n\in\N$. Thus, $(\widehat{\mu_n})_{n\in\N}$ does not converge vaguely to $0$.
\end{example}

A concrete characterisation is given by the next proposition. Recall here first that any Fourier transformable measure is a tempered measure.

\begin{proposition}
Let $(\mu_n)_{n\in\N}$ be a sequence of Fourier transformable measures on $\R^d$, and let $\mu$ be a Fourier transformable measure on $\R^d$. Then, $(\widehat{\mu_n})_{n\in\N}$ converges vaguely to $\widehat{\mu}$ if and only if
\begin{enumerate}
\item[(i)] the set $\{\widehat{\mu_n}\, :\, n\in\N\}$ is vaguely bounded, and
\item[(ii)] $\limn \mu_n(\widehat{f}) = \mu(\widehat{f})$ for all $f\in\Cc^{\infty}(\R^d)$.
\end{enumerate}
\end{proposition}
\begin{proof}
This is a consequence of Corollary~\ref{coro:chara_vc}.
\end{proof}

\begin{theorem}  \label{thm:crit_td}
Let $(\mu_n)_{n\in\N}$ be a sequence of Fourier transformable measures on $\R^d$, and let $\mu$ be a Fourier transformable measure on $\R^d$ such that
\begin{enumerate}
\item[(i)] $(\mu_n)_{n\in\N}$ converges to $\mu$ in the tempered distribution topology,
\item[(ii)] $\{\widehat{\mu_n}\, :\, n\in\N\}$ is vaguely bounded.
\end{enumerate}
Then, $(\widehat{\mu_n})_{n\in\N}$ converges vaguely to $\widehat{\mu}$.
\end{theorem}
\begin{proof}
This follows from the previous proposition.
\end{proof}

So far, we have always assumed that $\mu$ is a Fourier transformable measure. Next, we want to find sufficient conditions which also imply that $\mu$ is Fourier transformable. In order to do so, we need some preparation.

\begin{lemma}\label{L1}
For all $\mu \in \cM^\infty(\R^d)$ and all $f \in \Cz(\R^d)$ with $\| (1+|\bullet|^2)^d\, f \|_\infty < \infty$, we have
\[
| \mu(f)| \leqslant C\, \| (1+|\bullet|^2)^d\, f \|_\infty\, \| \mu \|_{[-\frac{1}{2},\frac{1}{2}]^d} \,,
\]
where
\[
C:=   \sum_{n \in \Z^d}\sup_{x\, \in\, n+[-\frac{1}{2},\frac{1}{2}]^d}  \frac{1}{(1+|x|^2)^d} < \infty \,.
\]
\end{lemma}
\begin{proof}
First, one has
\begin{align*}
| \mu(f)|
    &= \left| \int_{\R^d} f(x)\ \dd \mu(x) \right| \leqslant \sum_{n \in \Z^d}
       \left| \int_{n+[-\frac{1}{2},\frac{1}{2}]^d} f(x)\ \dd \mu(x) \right|  \\
    &=\sum_{n \in \Z^d} \left| \int_{n+[-\frac{1}{2},\frac{1}{2}]^d}
       \frac{1}{(1+|x|^2)^d}\, (1+|x|^2)^d\, f(x)\ \dd \mu(x) \right|  \\
    &\leqslant \sum_{n \in \Z^d}  \int_{n+[-\frac{1}{2},\frac{1}{2}]^d}\left|
       \frac{1}{(1+|x|^2)^d}\, (1+|x|^2)^d\, f(x)\right|\ \dd |\mu|(x) \\
    &\leqslant \sum_{n \in \Z^d} \| (1+|\bullet|^2)^d\, f \|_\infty  \int_{n+[-
      \frac{1}{2},\frac{1}{2}]^d} \frac{1}{(1+|x|^2)^d}\ \dd
       |\mu|(x)   \\
    &\leqslant \sum_{n \in \Z^d} \| (1+|\bullet|^2)^d\, f \|_\infty \left(\sup_{
       x\,\in  \,n+[-\frac{1}{2},\frac{1}{2}]^d}  \frac{1}{(1+|x|^2)^d} \right)
       |\mu|\Big(n+\big[-\frac{1}{2},\frac{1}{2}\big]^d\Big)   \\
    &\leqslant C\, \| (1+|\bullet|^2)^d\, f \|_\infty\,  \| \mu \|_{[-\frac{1}{2},
       \frac{1}{2}]^d}  \,,
\end{align*}
where
\[
C:=   \sum_{n \in \Z^d}\sup_{x\, \in\, n+[-\frac{1}{2},\frac{1}{2}]^d}  \frac{1}{(1+|x|^2)^d} < \infty \,.
\]
The claim follows.
\end{proof}

\begin{proposition} \label{prop:vtdd}
Let $(\mu_n)_{n\in\N}$ be a sequence of measures on $\R^d$, and let $\mu$ be a measure on $\R^d$. Consider the following statements:
\begin{itemize}
  \item [(i)] $(\mu_n)_{n\in\N}$ converges vaguely to $\mu$.
  \item [(ii)] $(\mu_n)_{n\in\N}$ converges to $\mu$ in the tempered distribution topology.
  \item [(iii)] $(\mu_n)_{n\in\N}$ converges to $\mu$ in the distribution topology.
\end{itemize}
Then,
\begin{enumerate}
\item[$\bullet$] (ii) implies (iii),
\item[$\bullet$] (iii) implies (i) if $(\mu_n)_{n\in\N}$ is vaguely bounded,
\item[$\bullet$] (i) implies (ii) if $\{\mu_n\, :\, n\in\N\}\cup \{\mu\}$ is equi translation bounded.
\end{enumerate}
\end{proposition}
\begin{proof}
(i)$\implies$(ii) Since $\{ \mu_n \, :\, n\in\N\} \cup \{ \mu \}$ is a set of equi translation bounded measures, we can find a constant $C_1>0$ such that
\begin{equation}  \label{eq:etb1}
 \| \mu \|_{[-\frac{1}{2},\frac{1}{2}]^d}   < C_1  \quad \text{ and } \quad
   \| \mu_n \|_{[-\frac{1}{2},\frac{1}{2}]^d}   < C_1 \,,
\end{equation}
for all $n\in\N$. Let $\eps >0$. Since $f \in \cS(\R^d)$, we can then find some $g \in \Cc(\R^d)$ such that
\begin{equation} \label{eq:etb2}
\| (1+|\bullet|^2)^d\, (f-g) \|_\infty < \frac{\eps}{3 \,  C \, C_1} \,.
\end{equation}
Then, by Eqs.~\eqref{eq:etb1}, \eqref{eq:etb2} and Lemma~\ref{L1}, we have
\[
\left| \mu (f-g) \right| < \frac{\eps}{3}   \quad \text{ and } \quad
\left| \mu_n (f-g) \right| < \frac{\eps}{3} \,,
\]
for all $n\in\N$.
Finally, since $(\mu_n)_{n\in\N}$ converges to $\mu$ in the vague topology and $g \in \Cc(\R^d)$, we can find some $N\in\N$ such that, for all $n >N$, we have
\[
\left| \mu_n(g) -\mu(g) \right| <\frac{\eps}{3} \,.
\]
This yields
\[
\left| \mu_n(f) -\mu(f) \right| < \eps
\]
for all $n >N$.

\smallskip

\noindent (ii)$\implies$(iii) Trivial.

\smallskip

\noindent (iii)$\implies$(i) Since $(\mu_n)_{n\in\N}$ is vaguely bounded, to prove that $\mu_n \to \mu$ vaguely it suffices to show that each vague limit point of $(\mu_n)_{n\in\N}$ is equal to $\mu$.

Let $\nu$ be a vague limit point of $(\mu_n)_{n\in\N}$. Then, there exists some subsequence $(n_k)_{k\in\N}$ such that $(\mu_{n_k})_{k\in\N}$ converges vaguely to $\nu$. It follows that, for all $f \in \Cc^\infty(\R^d)$, we have
\begin{displaymath}
\mu(f)= \lim_{n\to\infty} \mu_n(f)= \lim_{k\to\infty} \mu_{n_k}(f)= \nu(f) \,.
\end{displaymath}
Since $\mu =\nu$ on $\Cc^\infty(\R^d)$ and $\Cc^\infty(\R^d)$ is dense in $\Cc(\R^d)$, we get $\mu=\nu$.
\end{proof}

\begin{coro}
Let $(\mu_n)_{n\in\N}$ be a sequence of measures on $\R^d$ which is equi translation bounded, and let $\mu$ be a measure on $\R^d$. Then, the following statements are equivalent:
\begin{itemize}
  \item [(i)] $(\mu_n)_{n\in\N}$ converges vaguely to $\mu$.
  \item [(ii)] $(\mu_n)_{n\in\N}$ converges to $\mu$ in the tempered distribution topology.
  \item [(iii)] $(\mu_n)_{n\in\N}$ converges to $\mu$ in the distribution topology.
\end{itemize}
\end{coro}

\begin{coro} \label{coro:hilfe1}
Let $(\mu_n)_{n\in\N}$ be a sequence of measures on $\R^d$, and let $\mu$ be a Fourier transformable measure on $\R^d$ such that
\begin{enumerate}
  \item [(i)] $(\mu_n)_{n\in\N}$ converges vaguely to $\mu$,
  \item[(ii)] $\{ \mu_n\, :\, n\in\N\}$ is a set of equi translation bounded measures.
\end{enumerate}
Then, $(\widehat{\mu_n})_{n\in\N}$ converges to $\widehat{\mu}$ in the tempered distribution topology.
\end{coro}

\begin{example}\label{EX1}
Let $(a_n)_{n\in\N}$ be any sequence of non-zero real numbers which converges to $0$. Let $\mu_n =\frac{1}{a_n} (\delta_{a_n}-\delta_0) $, for all $n\in\N$. Then, as tempered distributions, $(\mu_n)_{n\in\N}$ converges to the distribution
\begin{displaymath}
\mathcal{D}(f)=f'(0) \,.
\end{displaymath}
\end{example}

\begin{example} \label{ex:tempdis}
Let
\begin{displaymath}
\mu_n:= n \delta_{\frac{1}{n}}+n \delta_{-\frac{1}{n}} -2n \delta_0,\qquad n\in\N \,.
\end{displaymath}
Then, $(\mu_n)_{n\in\N}$ converges in the tempered distribution topology to $0$ but it is not vaguely convergent.
\end{example}
\begin{proof}
For all $n\in\N$,
\begin{displaymath}
\mu_n= \left( \frac{1}{1/n} (\delta_{1/n}-\delta_0) \right)-\left( \frac{1}{-1/n} (\delta_{-1/n}-\delta_0) \right)
\end{displaymath}
is the difference of two sequences of measures, which converge to $\mathcal{D}$ by Example~\ref{EX1}. Therefore, $(\mu_n)_{n\in\N}$ converges in the tempered distribution topology to $\mathcal{D}-\mathcal{D}=0$.

Next, let
\begin{displaymath}
f(x)=
\begin{cases}
\sqrt{x} & \mbox{ if } 0 \leqslant x \leqslant 1, \\
2-x & \mbox{ if } 1 < x \leqslant 2, \\
0 &\mbox{otherwise}.
\end{cases}
\end{displaymath}
Then,
\begin{displaymath}
\mu_n(f)=\frac{n}{\sqrt{n}} \xrightarrow{n\to\infty} \infty  \,.
\end{displaymath}
This shows that $(\mu_n)_{n\in\N}$ cannot be vaguely convergent.
\end{proof}

Another consequence of Proposition~\ref{prop:vtdd} is the following proposition.

\begin{proposition}\label{ft td}
The Fourier transform, considered as a mapping from $(\cM^{\infty}(\R^d),\|\cdot\|_K)$ to $\cS'(\R^d)$, is continuous.
\end{proposition}
\begin{proof}
First note that the set $\{\mu_n\, :\, n\in\N\}\cup \{\mu\}$ is equi translation bounded because $(\mu_n)_{n\in\N}$ converges to $\mu$ with respect to $\|\cdot\|_K$. Now, by Proposition~\ref{prop:vtdd}, $(\mu_n)_{n\in\N}$ converges to $\mu$ in the tempered distribution topology. But this implies that $(\widehat{\mu_n})_{n\in\N}$ converges to $\widehat{\mu}$ in the tempered distribution topology.
\end{proof}

\begin{coro}
Let $(\mu_n)_{n\in\N}$ be a sequence of Fourier transformable measures in the space $(\cM^{\infty}(\R^d),\|\cdot\|_K)$ which converges to some measure $\mu$. If $(\widehat{\mu_n})_{n\in\N}$ has a vague cluster point, then the distributional Fourier transform of $\mu$ is a measure, and it is the cluster point of $(\widehat{\mu_n})_{n\in\N}$.
\end{coro}
\begin{proof}
Let $\nu$ be the vague cluster point of $(\widehat{\mu_n})_{n\in\N}$ and let $(n_k)_{k\in\N}$ be the subsequence such that
\begin{displaymath}
\lim_{k\to\infty} \widehat{\mu_{n_k}} = \nu \,.
\end{displaymath}
Then, the set $\{ \widehat{\mu_{n_k}}\, :\, k\in\N \}$ is vaguely bounded.

Since $\mu \in \cM^{\infty}(\R^d)$, it is a tempered distribution. Let $\phi$ be its Fourier transform as a tempered distribution. Then, by Proposition~\ref{ft td}, $\widehat{\mu_n} \to \phi$ in the tempered distribution topology.
It follows that $\left(\widehat{\mu_{n_k}}\right)_{k\in\N}$ converges to both $\nu$ and $\phi$ in the distribution topology, which shows that $\nu= \phi$ as distribution.
\end{proof}

Now, we are ready to state and prove the main result of this section.

\begin{theorem}\label{T3}
Let $(\mu_n)_{n\in\N}$ be a sequence of Fourier transformable measure on $\R^d$, and let $\mu$ be a measure on $\R^d$ such that
\begin{enumerate}
  \item [(i)] $(\mu_n)_{n\in\N}$ converges vaguely to $\mu$,
  \item[(ii)] $\mu$ is tempered,
  \item[(iii)] $\{ \widehat{\mu_n}\, :\, n\in\N\}$ is a set of equi translation bounded measures.
\end{enumerate}
Then, $\mu$ is Fourier transformable and, in the vague topology, we have
\[
\limn  \widehat{\mu_n} = \widehat{\mu} \,.
\]
\end{theorem}
\begin{proof}
Since $\mu$ is tempered, it is Fourier transformable as a tempered distribution. Let $\varphi$ denote its Fourier transform in the sense of tempered distribution.

Next, by (iii), there exists a compact set $K \subseteq \R^d$ with non-empty interior and a constant $C >0$ such that
\[
\widehat{\mu_n} \in \cM_{C,K}:= \{ \mu \in \cM^{\infty}(\R^d) \ :\ \| \mu \|_K \leqslant C \} \,,
\]
for all $n\in\N$.
We will use below the fact that $\cM_{C,K}$ is vaguely compact and metrisable \cite{BL}. It follows that $(\widehat{\mu_n})_{n\in\N}$ has vague cluster points, which are all translation bounded.
Let us observe first that, if $\nu$ is any vague cluster point of $(\widehat{\mu_n})_{n\in\N}$, then there exists a subsequence $(n_k)_{k\in\N}$ such that $(\widehat{\mu_{n_k}})_{k\in\N}$ converges vaguely to $\nu$. Then, by Proposition~\ref{prop:vtdd}, for all $f \in \Cc^\infty(\R^d)$, we have
\[
\nu(\reallywidecheck{f})=\lim_{k\to\infty} \widehat{\mu_{n_k}}(\reallywidecheck{f}) \,.
\]
Therefore, for all $g \in \Cc^\infty(\R^d) \subseteq KL(\R^d)$, we have
\begin{equation}\label{EQ2}
\nu(\reallywidecheck{g})=\lim_{k\to\infty} \widehat{\mu_{n_k}}(\reallywidecheck{g}) = \lim_{k\to\infty} \mu_{n_k}(g) = \mu(g) \,.
\end{equation}
We split the rest of the proof into steps.

\medskip

\noindent {\it Step 1:} We show that $(\widehat{\mu_n})_{n\in\N}$ is vaguely convergent to some measure $\nu$.

By a standard argument, as $\cM_{C,K}$ is vaguely compact, we only need to show that any two vague cluster points of this sequence are equal. Let $\nu_1, \nu_2$ be two such vague cluster points.
Then, by Eq.~\eqref{EQ2}, we have
\[
\nu_1(\reallywidecheck{g}) = \mu(g)= \nu_2(\reallywidecheck{g})
\]
for all $g \in \Cc^\infty(\R^d) \cap KL(\R^d)$.
Now, since $\nu_1,\nu_2$ are translation bounded, they are tempered as distributions \cite[Sec. 7]{ARMA1}.
Since the set $\{\reallywidecheck{g} \ :\ g \in \Cc^\infty(\R^d) \cap KL(\R^d) \}=\{ \reallywidecheck{g} \ :\ g \in \Cc^\infty(\R^d)\}$ is dense in $\cS(\R^d)$, it follows that $\nu_1=\nu_2$ as tempered distributions.
In particular, for all $h \in \Cc^\infty(\R^d)\subseteq \cS(\R^d)$, we have
\[
\nu_1(h)=\nu_2(h) \,.
\]
Therefore, the measures $\nu_1, \nu_2$ agree on the set $ \Cc^\infty(\R^d)$ which is dense in $\Cc(\R^d)$. Thus, $\nu_1=\nu_2$ as claimed, which shows that  $(\widehat{\mu_n})_{n\in\N}$ is vaguely convergent. Let $\nu$ be the vague limit of this sequence.

\medskip

\noindent {\it Step 2:} We show that $\phi=\nu$.

Let $g \in  \Cc^\infty(\R^d) \subseteq KL(\R^d)$. Then, by Eq.~\eqref{EQ2}, we have
\[
\nu(\reallywidecheck{g}) = \mu(g) \,.
\]
Moreover, as $g \in \Cc^\infty(\R^d) \subseteq \cS(\R^d)$, we also have $\mu(g) = \varphi(\reallywidecheck{g})$. This shows that
\begin{equation*}
\nu(\reallywidecheck{g})= \varphi(\reallywidecheck{g}) \qquad \text{ for all } g \in  \Cc^\infty(\R^d) \cap KL(\R^d) \,.
\end{equation*}
Therefore, since $\nu$ is translation bounded, it is tempered and, by the standard density argument, we get that $\varphi=\nu$ as tempered distributions.

This shows that the Fourier transform $\varphi$ as a tempered distribution is a translation bounded measure. Therefore, by \cite[Thm.~5.1]{NS12}, $\mu$ is Fourier transformable as a measure and
\[
\widehat{\mu}=\varphi=\nu \,.
\]
\end{proof}

\begin{coro}\label{T2}
Let $\mu_n, \mu$ be Fourier transformable measures in $\cM(\R^d)$, for all $n\in\N$, with the following properties:
\begin{enumerate}
  \item [(i)] $(\mu_n)_{n\in\N}$ converges vaguely to $\mu$,
  \item [(ii)] $\{ \widehat{\mu_n}\, :\, n\in\N\}$ is a set of equi translation bounded measures.
\end{enumerate}
Then, in the vague topology, we have
\[
\limn  \widehat{\mu_n} = \widehat{\mu} \,.
\]
\end{coro}
\begin{proof}
This follows from the previous theorem, since every Fourier transformable measure is tempered.
\end{proof}

However, property (ii) is not necessary for the convergence of $(\widehat{\mu_n})_{n\in\N}$ as the next example shows (also compare the next theorem).

\begin{example}
Consider the measures
\[
\mu_n := 1_{[-n,n]}\, \lm, \qquad n\in\N \,.
\]
In this case, $(\mu_n)_{n\in\N}$ converges vaguely to $\mu:=\lm$, and we have
\[
\widehat{\mu_n} =  2n \operatorname{sinc}(2\pi n\bullet) \, \lm \,.
\]
The sequence $(\widehat{\mu_n})_{n\in\N}$ is not equi translation bounded because
\begin{align*}
\|\widehat{\mu_n}\|_{[-1,1]}
    &= 2n \sup_{t\in\R} \int_{[t-1,t+1]} |\operatorname{sinc}(2\pi nx)|\ \dd x \\
    &= \frac{1}{\pi} \sup_{t\in\R} \int_{[2\pi n(t-1),2\pi n(t+1)]}
       |\operatorname{sinc}(y)|\ \dd y  \\
    &\geqslant \frac{1}{\pi} \int_{[-2\pi n,2\pi n]}
       |\operatorname{sinc}(y)|\ \dd y  \xrightarrow{n\to\infty} \infty \,.
\end{align*}
Still, $(\widehat{\mu_n})_{n\in\N}$ converges vaguely to $\widehat{\mu}=\delta_0$.
\end{example}

\begin{theorem} \label{thm:mainvag2}
Let $(\mu_n)_{n\in\N}$ be a sequence of Fourier transformable measures on $\R^d$, and let $\mu$ be a measure on $\R^d$ such that
\begin{enumerate}
\item [(i)] $(\mu_n)_{n\in\N}$ converges vaguely to $\mu$,
\item[(ii)] $\{ \mu_n\, :\, n\in\N\}$ is a set of equi translation bounded measures.
\end{enumerate}
Then, $\mu$ is Fourier transformable, and $(\widehat{\mu_n})_{n\in\N}$ converges vaguely to $\widehat{\mu}$ if and only if the set $\{\widehat{\mu_n}\, : \, n\in\N\}$ is vaguely bounded.
\end{theorem}
\begin{proof}
First, assume that $\{\widehat{\mu_n}\, : \, n\in\N\}$ is vaguely bounded. By Proposition~\ref{prop:vtdd}, the sequence $(\mu_n)_{n\in\N}$ also converges to $\mu$ in the tempered distribution topology. This implies that $\widehat{\mu}$ exists as tempered distribution such that $(\widehat{\mu_n})_{n\in\N}$ converges to $\widehat{\mu}$ in the sense of tempered distributions.

Next, $\{\widehat{\mu_n}\, :\, n\in\N\}$ is vaguely compact by Proposition~\ref{P2}. Hence, there is at least one vague cluster point, say $\nu$. So, there exists a subsequence $(n_k)_{k\in\N}$ such that
\[
\lim_{k\to\infty} \widehat{\mu_{n_k}}(f) = \nu(f), \qquad \text{ for all } f\in\Cc(\R^d) \,.
\]
This and the fact that
\[
\lim_{n\to\infty} \widehat{\mu_n}(\phi) = \widehat{\mu}(\phi),  \qquad \text{ for all } f\in\cS(\R^d)  \,,
\]
imply that $\nu$ and $\widehat{\mu}$ coincide on $\Cc^{\infty}(\R^d)$, which is a dense subset of $\Cc(\R^d)$. This finishes the proof.

\medskip

On the other hand, assume that $\mu$ is Fourier transformable and that $(\widehat{\mu_n})_{n\in\N}$ converges vaguely to $\widehat{\mu}$. This trivially implies that $\{\widehat{\mu_n}\, : \, n\in\N\}$ is vaguely bounded.
\end{proof}

So far, we had always assumed that $(\mu_n)_{n\in\N}$ converges vaguely to $\mu$. But this is not necessary for $(\widehat{\mu_n})_{n\in\N}$ to converge to $\widehat{\mu}$, as the next example shows.

\begin{example}\label{ex417}
Consider the sequence of measures $(\mu_n)_{n\in\N}$ from Example~\ref{ex:tempdis}, i.e.
\begin{displaymath}
\mu_n:= n \delta_{\frac{1}{n}}+n \delta_{-\frac{1}{n}} -2n \delta_0,\qquad n\in\N \,.
\end{displaymath}
We saw that it doesn't converge vaguely to $0$. However, notice that it converges to $0$ in the tempered distribution topology. Also, we have
\[
\widehat{\mu_n} = n \e^{-2\pi\im \frac{1}{n}\bullet} \,\lm + n \e^{2\pi\im \frac{1}{n}\bullet} \,\lm -2n \lm,\qquad n\in\N \,.
\]
Therefore, $(\widehat{\mu_n})_{n\in\N}$ converges vaguely to $0$.
\end{example}

\section{Continuity in the product topology} \label{sec:product}

Next, we want to investigate under which circumstances we can obtain a stroger kind of convergence of $(\widehat{\mu_{\alpha}})_{\alpha}$. Note that the assumptions from Theorem~\ref{T3} are not sufficient to guarantee convergence in the norm topology or in the product topology, as the next example shows.

\begin{example}
Consider the sequence of measures $(\mu_n)_{n\in\N}$ with
\[
\mu_n := \e^{-2\pi\im n\bullet}\, \lm \,.
\]
The Riemann--Lebesgue lemma implies that $(\mu_n)_{n\in\N}$ converges vaguely to $\mu=0$. Obviously, $\mu$ is tempered/Fourier transformable, and $\widehat{\mu} = 0$. Moreover, we have $\widehat{\mu_n} = \delta_{-n}$. So, if $K$ is any compact set in $\R$, we obtain
\[
\| \widehat{\mu_n}\|_K = \sup_{t\in\R} |\widehat{\mu_n}|(t+K) = \sup_{t\in\R} \delta_{-n}(t+K) = 1 \,.
\]
Consequently, the sequence $(\mu_n)_{n\in\N}$ satisfies the assumptions from Theorem~\ref{T3}, and $(\widehat{\mu_n})_{n\in\N}$ converges vaguely to $0$.

However, $(\widehat{\mu_n})_{n\in\N}$ does not converge in the product topology (hence, not in the norm topology, either) because
\[
\| g*\widehat{\mu_n} - g*\widehat{\mu} \|_{\infty}
    = \| g*\delta_{-n}\|_{\infty} = \|g\|_{\infty}
\]
for all $g\in\Cc(\R)$.
\end{example}

If we replace (i) in Therorem~\ref{T3} by a stronger property, the sequence $(\widehat{\mu_n})_{n\in\N}$ converges to $\widehat{\mu}$ in the product topology.

\begin{theorem} \label{thm:prodtop}
Let $(\mu_n)_{n\in\N}$ be a sequence of Fourier transformable measures on $\R^d$, and let $\mu$ be a measure on $\R^d$ such that
\begin{enumerate}
  \item [(i)] $\left(\mu_n(\e^{-2\pi\im t\bullet}f)\right)_{n\in\N}$ converges to $\mu(\e^{-2\pi\im t\bullet}f)$ uniformly in $t$, for all $f\in\cS(\R^d)$,
  \item[(ii)] $\{\mu_n\, :\, n\in\N\}$ is vaguely bounded,
  \item[(iii)] $\mu$ is tempered,
  \item[(iv)] $\{ \widehat{\mu_n}\, :\, n\in\N\}$ is a set of equi translation bounded measures.
\end{enumerate}
Then, $\mu$ is Fourier transformable and, in the product topology, we have
\[
\limn  \widehat{\mu_n} = \widehat{\mu} \,.
\]
\end{theorem}
\begin{proof}
First note that (i) and (ii) imply that $(\mu_n)_{n\in\N}$ converges vaguely to $\mu$, see Theorem~\ref{thm:crit_td}. Thus, by Theorem~\ref{T3}, $\mu$ is Fourier transformable and
\[
\widehat{\mu_n} \to \widehat{\mu} \qquad \text{vaguely}\,.
\]

Next, by (i), one has
\begin{align*}
\| \phi*\widehat{\mu_n} - \phi*\widehat{\mu} \|_{\infty}
    &= \sup_{t\in \R^d} \left| \int_{\R^d} (T_t\phi^{\dagger})(x) \
       \dd(\widehat{\mu_n -\mu})(x) \right|  \\
    &= \sup_{t\in \R^d} \left| \int_{\R^d} \widehat{(T_t\phi^{\dagger})}(x)
       \ \dd(\mu_n -\mu)(x) \right|   \\
    &= \sup_{t\in \R^d} \left| \int_{\R^d} \e^{-2\pi\im tx}\,
       \reallywidecheck{\phi}(x) \ \dd(\mu_n -\mu)(x) \right|  \\
    &\xrightarrow{\text{(i)}} 0
\end{align*}
for every $\phi\in\Cc^{\infty}(\R^d)$. Since $\reallywidecheck{\phi}\in\cS(\R^d)$ and $\Cc^{\infty}(\R^d)$ is dense in $\Cc(\R^d)$, the claim follows.
\end{proof}

\begin{lemma} \label{lem:A}
Let (I) denote the property from Theorem~\ref{thm:prodtop}(i).
\begin{enumerate}
\item[(a)] (I) implies that $(\mu_n)_{n\in\N}$ converges vaguely to $\mu$.
\item[(b)] (I) does not imply that $(\mu_n)_{n\in\N}$ converges in the product topology to $\mu$.
\item[(c)] If $(\mu_n)_{n\in\N}$ converges in the product topology to $\mu$, then  (I) does not hold in general.
\item[(d)] If $(\mu_n)_{n\in\N}$ satisfies
\[
\limn \int_{\R^d} |f(x)|\ \dd |\mu_n-\mu|(x) = 0 \qquad \text{ for all } f\in\cS(\R^d)
\]
for some measure $\mu$, then it also satisfies (I).
\item[(e)] If $(\mu_n)_{n\in\N}$ converges in the norm topology to $\mu$, then it satisfies the property from (d), hence (I).
\end{enumerate}
\end{lemma}
\begin{proof}
(a) This is included in the proof of Theorem~\ref{thm:prodtop}.

\medskip

\noindent (b) Consider the sequence $(\mu_n)_{n\in\N}$ with $\mu_n=\delta_{-n}$, for all $n\in\N$. This sequence satisfies (I) with $\mu=0$ because
\[
\lim_{n\to\infty} \sup_{t\in\R} \left| \int_{\R} \e^{-2\pi\im tx} f(x)\ \dd(\mu_n-\mu)(x) \right| = \limn \sup_{t\in\R} |\e^{2\pi\im tn}\, f(-n)| = \limn |f(-n)| = 0
\]
for all $f\in\cS(\R)$. However, $(\mu_n)_{n\in\N}$ does not converge to $\mu$ in the product topology, since
\[
\|\phi*\mu_n-\phi*\mu\|_{\infty} = \|\phi*\delta_{-n}\|_{\infty} = \|\phi\|_{\infty},
\]
for all $n\in\N$ and $\phi\in\Cc(\R)$.

\medskip

\noindent (c)  Consider the sequence $(\mu_n)_{n\in\N}$ with $\mu_n=\delta_{\frac{1}{n}}$, for all $n\in\N$. Then, $(\mu_n)_{n\in\N}$ converges to $\mu=\delta_0$ in the product topology because
\[
\limn \|\phi*\mu_n-\phi*\mu\|_{\infty} = \limn \|T_{\frac{1}{n}}\phi -\phi\|_{\infty} =0
\]
for all $\phi\in\Cc(\R)\subseteq \Cu(\R)$. On the other hand, (I) is not satisfied, since
\[
\sup_{t\in\R} \left|\int_{\R} \e^{-2\pi\im tx} f(x)\ \dd(\mu_n-\mu)(x) \right| =  \sup_{t\in\R} |\e^{2\pi\im tn^{-1}} f(n^{-1})-f(0)| \xrightarrow{n\to\infty} 2\, |f(0)|
\]
for all $\phi\in\cS(\R)$.

\medskip

\noindent (d) This follows from
\[
\sup_{t\in\R^d} \left| \int_{\R^d} \e^{-2\pi\im tx} f(x)\ \dd(\mu_n-\mu)(x) \right| \leqslant  \int_{\R^d} |f(x)| \ \dd|\mu_n-\mu|(x)
\]
for all $f\in\cS(\R^d)$.

\medskip

\noindent (e) Set $s:=\sum_{k\in\Z^d} \sup_{x\in [0,1]^d+k} \frac{1}{(1+|x|^2)^d} < \infty$. Let $f\in\cS(\R^d)$. Then, there is a constant $c>0$ such that
\begin{align*}
\int_{\R^d} |f(x)|\ \dd|\mu_n-\mu|(x)
    &= \sum_{k\in\Z^d} \int_{[0,1]^d+k} |f(x)|\ \dd|\mu_n-\mu|(x)  \\
    &\leqslant c \sum_{k\in\Z^d} \int_{[0,1]^d+k} \frac{1}{(1+|x|^2)^d} \
      \dd|\mu_n-\mu|(x)  \\
    &\leqslant c\, \|\mu_n-\mu\|_{[0,1]^d} \sum_{k\in\Z^d} \sup_{x\in [0,1]^d+k}
       \frac{1}{(1+|x|^2)^d}  \\
    &= c\,s \, \|\mu_n-\mu\|_{[0,1]^d} \,,
\end{align*}
which implies the claim.
\end{proof}

\begin{coro}
Let $(\mu_n)_{n\in\N}$ be a sequence of Fourier transformable measures on $\R^d$, and let $\mu$ a measure on $\R^d$ such that
\begin{enumerate}
  \item [(i)] $\mu_n \to \mu$ in the norm topology,
  \item[(ii)] $\mu$ is tempered,
  \item[(iii)] $\{ \widehat{\mu_n}\, :\, n\in\N\}$ is a set of equi translation bounded measures.
\end{enumerate}
Then, $\mu$ is Fourier transformable and, in the product topology, we have
\[
\limn  \widehat{\mu_n} = \widehat{\mu} \,.
\]
\end{coro}
\begin{proof}
This is an immediate consequence of Theorem~\ref{thm:prodtop} and Lemma~\ref{lem:A}(e).
\end{proof}

If we make use of Theorem~\ref{thm:mainvag2} instead of Theorem~\ref{T3}, we obtain a different criterion for the convergence in the product topology.

\begin{theorem}
Let $(\mu_n)_{n\in\N}$ be a sequence of Fourier transformable measures on $\R^d$, and let $\mu$ be a measure on $\R^d$ such that
\begin{enumerate}
  \item [(i)] $\left(\mu_n(\e^{-2\pi\im t\bullet}f)\right)_{n\in\N}$ converges to $\mu(\e^{-2\pi\im t\bullet}f)$ uniformly in $t$, for all $f\in\cS(\R^d)$,
  \item[(ii)] the set $\{\mu_n\, :\, n\in\N\}$ is a set of equi translation bounded measures,
  \item[(iii)] the set $\{ \widehat{\mu_n}\, :\, n\in\N\}$ is vaguely bounded.
\end{enumerate}
Then, $\mu$ is Fourier transformable and, in the product topology, we have
\[
\limn  \widehat{\mu_n} = \widehat{\mu} \,.
\]
\end{theorem}

We can also apply Theorem~\ref{thm:crit_td}.

\begin{theorem}
Let $(\mu_n)_{n\in\N}$ be a sequence of Fourier transformable measures on $\R^d$, and let $\mu$ be a Fourier transformable measure on $\R^d$ such that
\begin{enumerate}
\item [(i)] $\left(\mu_n(\e^{-2\pi\im t\bullet}f)\right)_{n\in\N}$ converges to $\mu(\e^{-2\pi\im t\bullet}f)$ uniformly in $t$, for all $f\in\cS(\R^d)$,
\item[(iii)] the set $\{ \widehat{\mu_n}\, :\, n\in\N\}$ is vaguely bounded.
\end{enumerate}
Then, in the product topology, we have
\[
\limn  \widehat{\mu_n} = \widehat{\mu} \,.
\]
\end{theorem}

\section{Continuity in the norm topology} \label{sec:norm}

Note that, in general, the assumptions from the last corollary are not sufficient to ensure that the sequence $(\widehat{\mu_n})_{n\in\N}$ converges to $\widehat{\mu}$ in the norm topology.

\begin{example}
Consider the sequence $(\mu_n)_{n\in\N}$ with $\mu_n=\e^{2\pi\im \frac{1}{n}\bullet}\lm$. Because of
\[
\|\mu_n-\lm \|_{[0,1]} = \sup_{t\in\R} \int_{[0,1]+t} |\e^{2\pi\im \frac{1}{n}x}-1|\ \dd x \leqslant \int_{[0,42]} |\e^{2\pi\im \frac{1}{n}x}-1|\ \dd x \xrightarrow{n\to\infty} 0 \,,
\]
the sequence $(\mu_n)_{n\in\N}$ converges to $\mu=\lm$ in the norm topology, and $\mu$ is obviously tempered/Fourier transformable. Moreover, since $\widehat{\mu_n}= \delta_{\frac{1}{n}}$, we have
\[
\|\widehat{\mu_n}\|_{[0,1]} = \sup_{t\in\R} |\delta_{\frac{1}{n}}|(t+[0,1]) =1 \,,
\]
for all $n\in\N$. Therefore, $(\mu_n)_{n\in\N}$ satisfies the assumptions from the previous corollary, and $(\widehat{\mu_n})_{n\in\N}$ converges to $\widehat{\mu}=\delta_0$ in the product topology. However, $(\widehat{\mu_n})_{n\in\N}$ does not converge to $\widehat{\mu}$ in the norm topology because
\[
\|\widehat{\mu_n} - \widehat{\mu} \|_{[0,1]}  = \sup_{t\in\R} |\delta_{\frac{1}{n}}-\delta_0|(t+[0,1]) = 2 \,,
\]
for all $n\in\N$.
\end{example}

On the other hand, if the sequence $(\mu_n)_{n\in\N}$ does not satisfy (I), the sequence $(\widehat{\mu_n})_{n\in\N}$ can still converge in the norm topology.

\begin{example}
Consider the sequence $(\mu_n)_{n\in\N}$ with $\mu_n=\delta_{\frac{1}{n}}$, for all $n\in\N$. Then, $(\mu_n)_{n\in\N}$ does not satisfy (I) (see above), but $(\widehat{\mu_n})_{n\in\N}$ converges to $\widehat{\mu}=\lm$ in the norm topology (see above).
\end{example}

\begin{theorem}  \label{thm:norm_conv}
Let $(\mu_n)_{n\in\N}$ be a sequence of Fourier transformable measures on $\R^d$, and let $\mu$ be a measure on $\R^d$ such that
\begin{enumerate}
\item[(i)] $(\mu_n)_{n\in\N}$ converges vaguely to $\mu$,
\item[(ii)] $\mu$ is tempered,
\item[(iii)] $\left(\mu_n(\e^{-2\pi\im t\bullet}\widehat{f})\right)_{n\in\N}$ converges to $\mu(\e^{-2\pi\im t\bullet}\widehat{f})$ uniformly in $(t,f)\in\R^d \times \mathcal{F}_U^{\infty}(\R^d)$, where $U$ is a precompact open set and $\mathcal{F}_U^{\infty}(\R^d):=\{g\in\Cc^{\infty}(\R^d)\, :\, |g|\leqslant 1_U\}$.
\end{enumerate}
Then, $\mu$ is Fourier transformable and, in the norm topology, we have
\[
\limn  \widehat{\mu_n} = \widehat{\mu} \,.
\]
\end{theorem}
\begin{proof}
It follows from (iii) and \cite[Cor. 3.2]{SpSt} that the set $\{\widehat{\mu_n}\, :\, n\in\N\}$ is equi translation bounded because
\begin{align} \label{eq:norm1}
\| \widehat{\mu_n} \|_U
    &= \sup_{(t,f)\in\R^d\times \mathcal{F}_U(\R^d)} |\widehat{\mu_n} (T_tf)|
        \notag \\
    &=  \sup_{(t,f)\in\R^d\times \mathcal{F}_U^{\infty}(\R^d)} |
        \mu_n(\widehat{T_tf})|   \\
    &=  \sup_{(t,f)\in\R^d\times \mathcal{F}_U^{\infty}(\R^d)} |
        \mu_n(\e^{-2\pi\im t\bullet} \widehat{f})| \,.   \notag
\end{align}
Thus, we can apply Theorem~\ref{T3}: $\mu$ is Fourier transformable and $(\widehat{\mu_n})_{n\in\N}$ converges vaguely to $\widehat{\mu}$. It even converges in the norm topology because (similar to Eq.~\eqref{eq:norm1})
\[
\| \widehat{\mu_n}-\widehat{\mu}\|_U = \sup_{(t,f)\in\R^d\times \mathcal{F}_U^{\infty}(\R^d)} |(\mu_n-\mu)(\e^{-2\pi\im t\bullet} \widehat{f})|  \xrightarrow{n\to\infty} 0
\]
by (iii).
\end{proof}

\begin{coro}
Let $(f_n)_{n\in\N}$ be a sequence of functions in $L^2(\R^d)$ which converges to some $f$ in $L^2(\R^d)$. Consider the sequence $(\mu_n)_{n\in\N}$ with $\mu_n=f_n\, \lm$, for all $n\in\N$. Then, $(\widehat{\mu_n})_{n\in\N}$ converges to $\widehat{\mu} = \widehat{f}\, \lm$ in the norm topology.
\end{coro}
\begin{proof}
First, we will give a direct proof. Let $K$ be any compact set in $\R^d$. Then, the claim follows from
\begin{align*}
\|\widehat{\mu_n} - \widehat{\mu}\|_K
    &= \sup_{t\in\R^d} \int_{t+K} |\widehat{f_n}(x) - \widehat{f}(x)|\ \dd x  \\
    &\leqslant \sup_{t\in\R^d}\left( \int_{t+K} 1^2\ \dd x\right)^{\frac{1}{2}}
       \, \left( \int_{t+K} |\widehat{f_n}(x) - \widehat{f}(x)|^2\ \dd x
       \right)^{\frac{1}{2}}  \\
    &\leqslant |K|^{\frac{1}{2}}\, \|\widehat{f_n-f}\|_{L^2}  \\
    &= |K|^{\frac{1}{2}}\, \|f_n-f\|_{L^2} \,,
\end{align*}
by an application of H\"olders inequality and Plancherels theorem.

\medskip

\noindent Alternatively, it is not difficult to see that $(\mu_n)_{n\in\N}$ satisfies the assumptions of Theorem~\ref{thm:norm_conv}.
\end{proof}

\begin{coro}
Let $(\mu_n)_{n\in\N}$ be a sequence of Fourier transformable measures on $\R^d$, and let $\mu$ be a Fourier transformable measure on $\R^d$ such that
\[
 \left(\mu_n(\e^{-2\pi\im t\bullet}\widehat{f})\right)_{n\in\N} \text{ converges to } \mu(\e^{-2\pi\im t\bullet}\widehat{f}) \text{ uniformly in } (t,f)\in\R^d \times \mathcal{F}_U^{\infty}(\R^d).
\]
Then, in the norm topology, we have
\[
\limn  \widehat{\mu_n} = \widehat{\mu} \,.
\]
\end{coro}
\begin{proof}
This follows from Theorem~\ref{thm:crit_td}.
\end{proof}

\section{An application of Theorem~\ref{T3}}  \label{sec:application}

Next, we will make use of Theorem~\ref{T3} to give a necessary and sufficient condition for a measure $\mu$ on $\R^d$ to be Fourier transformable. Throughout this section, $(f_n)_{n\in\N}$ denotes an approximate identity for $(\Cu(\R^d),*)$ such that, for all $n\in\N$:
\begin{enumerate}
\item[$\bullet$] $f_n\in\Cc(\R^d)$,
\item[$\bullet$] there is a compact set $K\subseteq \R^d$ (independent of $n$) with $\supp(f_n)\subseteq K$,
\item[$\bullet$] $f_n\geqslant 0$,
\item[$\bullet$] $\int_{\R^d} f_n(x)\ \dd x =1$,
\item[$\bullet$] $f_n$ is positive definite.
\end{enumerate}
Such a sequence of functions exists by \cite[Thm. (44.20)]{HR}. Note that, by \cite[Lem. 3.6]{CRS}, one also has $\widehat{f_n} \in L^1(\R^d)$.

\begin{lemma} \label{lem:helpme1}
Let $\mu$ be a Fourier transformable measure, and let $\mu_n:=(f_n*\mu)\, \lm$, for all $n\in\N$. Then, $(\mu_n)_{n\in\N}$ converges vaguely to $\mu$.
\end{lemma}
\begin{proof}
The claim follows because $(f_n)_{n\in\N}$ is an approximate identity:
\begin{align*}
| \mu_n(\phi) - \mu(\phi)|
    &= \left| \int_{\R^d} \int_{\R^d} f_n(x-y)\, \phi(x)\ \dd x\ \dd \mu(y)
       - \int_{\R^d} \phi(y)\ \dd\mu(y) \right|  \\
    &\leqslant \int_{\R^d} \big| (f_n*\phi)(y) - \phi(y)\big|\ \dd |\mu|(y)   \\
    &\xrightarrow{n\to\infty} 0 \,,
\end{align*}
for all $\phi\in\Cc(\R^d)$, where we applied the dominated convergence theorem in the last step.
\end{proof}

\begin{lemma} \label{lem:helpme2}
Let $\mu$ be a Fourier transformable measure. Let $\nu_n:=\widehat{f_n} \, \widehat{\mu}$, for all $n\in\N$. Then, $\{ \nu_n\ :\ n\in\N\}$ is equi translation bounded.
\end{lemma}
\begin{proof}
First note that $\widehat{f_n}(0)=\int_{\R^d} f_n(x)\ \dd x=1$, for all $n\in\N$. Also note that $\widehat{f_n}$ is positive definite, since $f_n$ is non-negative. Hence, we have
\[
\big| \widehat{f_n}(x) \big| \leqslant \widehat{f_n}(0) =1
\]
for all $x\in \R^d$. Consequently, for a fixed compact set $K\subseteq \R^d$, one has
\[
\| \nu_n\|_K = \sup_{t\in\R^d} \int_{t+K} |\widehat{f_n}(x)|\ \dd|\widehat{\mu}| (x) \leqslant \| \widehat{\mu}\|_K  \,,
\]
for all $n\in\N$.
\end{proof}

Now, we can apply the previous two lemmas and Theorem~\ref{T3} to characterise Fourier transformable measures on $\R^d$.

\begin{theorem}
Let $\mu$ be a measure on $\R^d$. Then, $\mu$ is Fourier transformable if and only if it is tempered, and there is a sequence of finite measures $(\nu_n)_{n\in\N}$ that satisfies the following two properties:
\begin{enumerate}
\item[$\bullet$] the set $\{\nu_n\ :\ n\in\N\}$ is equi translation bounded,
\item[$\bullet$] $\reallywidecheck{\nu_n} = (f_n*\mu)\, \lm$.
\end{enumerate}
\end{theorem}
\begin{proof}
First assume that $\mu$ is a Fourier transformable meassure. Note that every Fourier transformable measure is tempered. Define $\nu_n:= \widehat{f_n} \, \widehat{\mu}$, for all $n\in\N$. By \cite[Prop. 3.9]{CRS}, $\nu_n$ is a finite measure, for all $n\in\N$, because
\[
\nu_n(\R^d) = \int_{\R^d} \widehat{f_n}(x)\ \dd\widehat{\mu}(x) \,.
\]
Moreover, \cite[Lem. 4.9.24]{TAO2} implies $\reallywidecheck{\nu_n} = (f_n*\mu)\, \lm$.

\bigskip

On the other hand, assume that $\mu$ is tempered, and that the two properties are satisfied. Let $\mu_n:= (f_n*\mu)\, \lm$, for all $n\in\N$. Then,
\begin{enumerate}
\item[1.] $\mu_n$ is Fourier transformable, for all $n\in\N$, because it is the Fourier transform of a finite measure, and every finite measure is twice Fourier transformable,
\item[2.] $(\mu_n)_{n\in\N}$ converges vaguely to $\mu$ by Lemma~\ref{lem:helpme1},
\item[3.] $\{\widehat{\mu_n}\ :\ n\in\N\}$ is equi translation bounded by Lemma~\ref{lem:helpme2}.
\end{enumerate}
Therefore, we can apply Theorem~\ref{T3}, which tells us that $\mu$ is Fourier transformable.
\end{proof}

\section{Positive definite measures} \label{sec:pd}

Now, we will consider positive definite measures on $\R^d$. For simplicity, we denote by $\cM_{\text{pd}}(\R^d)$ the set of positive definite measures on $\R^d$.
Let us start with the following known result (compare \cite{BF}).

\begin{proposition} \cite[Lem. 4.11.10]{TAO2}  \label{prop:pd_vauge}
Let $(\mu_{\alpha})_{\alpha}$ be a net of positive definite measures which converges vaguely to a measure $\mu$. Then, $\mu$ is positive definite and, in the vague topology, we have
\[
\lim_{\alpha} \widehat{\mu_{\alpha}} = \widehat{\mu} \,.
\]
\end{proposition}

In this section, we try to clarify why, when restricting to positive definite measures, the Fourier transform becomes continuous. The key is the following result.

\begin{lemma}\label{L2}
Let $K \subseteq \R^d$ be a compact set. Then,
\begin{itemize}
  \item [(i)] there exists some $f \in \Cc(\R^d)$ such that, for all $\mu \in \cM_{pd}(\R^d)$, we have
  \begin{displaymath}
\left| \widehat{\mu} \right| (K) \leqslant \mu(f) \,.
  \end{displaymath}
  \item [(ii)] there exists some compact $W \subseteq \R^d$ and some $C >0$ such that, for all $\mu \in \cM_{pd}(\R^d)$, we have
  \begin{displaymath}
\| \widehat{\mu} \|_K \leqslant C \left|\mu\right|(W) \,.
  \end{displaymath}
\end{itemize}
\end{lemma}
\begin{proof}
(i) By \cite{BF,MoSt}, there exists some $f \in \Cc(\R^d)$ such that $\reallywidecheck{f} \geqslant 1_K$. Now, using the fact that $\mu$ is Fourier transformable and the positivity of $\widehat{\mu}$, we have
\begin{displaymath}
\left| \widehat{\mu} \right| (K)=  \widehat{\mu} (K) \leqslant \widehat{\mu}(\reallywidecheck{f})=\mu(f) \leqslant \left| \mu(f) \right| \,.
\end{displaymath}

\noindent (ii) Let $f$ be as in (i). Let $W =\supp(f)$ and $C:= \| f\|_\infty$. We show that these satisfy the required condition. Indeed, for all $\mu \in \cM_{pp}(\R^d)$, we have
\begin{align*}
\| \widehat{\mu} \|_K
    & = \sup_{y  \in \R^d} \widehat{\mu}(T_y  K)\leqslant \sup_{y
       \in \R^d} \widehat{\mu}(T_y  \reallywidecheck{f})   \\
    &= \sup_{y  \in \R^d} \widehat{\mu}(\reallywidecheck{\e^{-2\pi\im
       y\bullet} f})=
       \sup_{y \in \R^d} \mu(\e^{-2\pi\im
       y\bullet}  f) = \sup_{y  \in \R^d}
       \int_{W} \e^{-2\pi\im
       yx}\, f(x)\ \dd \mu(x)  \\
    &\leqslant \sup_{y  \in \R^d} \int_{W} \left| \e^{-2\pi\im
       y x}\, f(x) \right| \ \dd \left| \mu\right|(x)
     \leqslant C\, \left|  \mu \right| (W) \,.
\end{align*}
\end{proof}

\smallskip

\begin{remark}\label{R1}
Lemma~\ref{L2} says that the Fourier transform
\begin{displaymath}
\widehat{}\, : (\cM_{\text{pd}}(\R^d), | \cdot |_W ) \to (\cM^\infty(\R^d), \| \cdot \|_K)
\end{displaymath}
is bounded, where $| \mu |_W := | \mu|(W)$.

Since $\cM_{\text{pd}}(\R^d)$ is not a vector space, this doesn't necessarily imply continuity, and the map above is actually not continuous. Indeed, if $\gamma$ is the autocorrelation of the Fibonacci model set, and $\{\gamma_n\, :\, n\in\N\}$ are the finite approximants, then $\gamma_n \to \gamma$ in the vague topology. Since $\gamma_n, \gamma$ are supported inside a common Delone set, it follows immediately that $| \gamma_n -\gamma |_W \to 0$.
But, since all $\widehat{\gamma_n}$ are absolutely continuous and $\widehat{\gamma}$ is pure point, we have
\[
\| \widehat{\gamma_n} -\widehat{\gamma} \|_K=\| \widehat{\gamma_n}\|_K+\|\widehat{\gamma} \|_K \geqslant \|\widehat{\gamma} \|_K >0 \,.
\]
This shows that $(\widehat{\gamma_n})_{n\in\N}$ does not converge to $\widehat{\gamma}$ in $ (\cM^\infty(\R^d), \| \cdot \|_K)$.
\end{remark}

\smallskip

While Remark~\ref{R1} emphasizes some of the issues around the continuity of the Fourier transform, Lemma~\ref{L2} can be used to show that vague convergence of positive definite measures implies equi translation boundedness of their Fourier transforms, which we saw earlier in the paper, suffices to deduce the vague convergence of the Fourier transforms.

Let us start with the following simple consequence.

\begin{coro}\label{cor1}
If $\cA \subseteq \cM_{pd}(\R^d)$ is vaguely bounded, then $\widehat{\cA}:=\{ \widehat{\mu} : \mu \in \cA \}$ is equi translation bounded.
\end{coro}

Corollary~\ref{coro:hilfe1} and Corollary~\ref{cor1} imply Proposition~\ref{prop:pd_vauge}.

\medskip

Next, we will state similar statements for the product and norm topology.

\begin{theorem}
Let $(\mu_{n})_{n\in\N}$ be a sequence of positive definite measures on $\R^d$, and let $\mu$ be a measure on $\R^d$ such that
\begin{enumerate}
\item[(i)] $\limn \mu_n(\e^{-2\pi\im t\bullet}f) = \mu(\e^{-2\pi\im t\bullet}f) \text{ uniformly in } t\in\R^d$, for all $f\in\cS(\R^d)$,
\item[(ii)] $\{\mu_n\, :\, n\in\N\}$ is vaguely bounded.
\end{enumerate}
Then, $\mu$ is positive definite and, in the product topology, we have
\[
\limn \widehat{\mu_{n}} = \widehat{\mu} \,.
\]
\end{theorem}
\begin{proof}
By Proposition~\ref{prop:pd_vauge}, $\mu$ is positive definite, and $(\widehat{\mu_{n}})_{n\in\N}$ converges vaguely to $\widehat{\mu}$. The rest follows as shown in the proof of Theorem~\ref{thm:prodtop}.
\end{proof}

\begin{theorem}
Let $(\mu_{n})_{n\in\N}$ be a sequence of positive definite measures on $\R^d$ which converges vaguely to a measure $\mu$. If
\[
\limn \mu_n\big(\e^{-2\pi \im t\bullet}\, \widehat{f}\big) = \mu\big(\e^{-2\pi\im t\bullet}\, \widehat{f}\big)  \qquad \text{ uniformly in } (t,f) \in\R^d\times \mathcal{F}_U^{\infty}(\R^d) \,,
\]
then $\mu$ is Fourier transformable, and we have
\[
\limn \widehat{\mu_n} = \widehat{\mu}
\]
in the norm topology.
\end{theorem}
\begin{proof}
By Proposition~\ref{prop:pd_vauge}, $\mu$ is positive definite, and $(\widehat{\mu_{n}})_{n\in\N}$ converges vaguely to $\widehat{\mu}$. The rest follows as shown in the proof of Theorem~\ref{thm:norm_conv}.
\end{proof}

\section{Some open questions}

In this section, we will look at some natural questions related to the (dis)continuity of the Fourier transform.

On the space of Fourier transformable translation bounded measures, the Fourier transform $\nu_n:=\widehat{\mu_n}$ of a vague convergent sequence $\mu_n \to \mu$ is either vague convergent or has a subsequence $(\nu_{n_k})_{k\in\N}$ without vague convergent subsequences. In particular, this subsequence cannot have any vaguely bounded subsequence, and hence there exists a compact set $W$ such that
\begin{displaymath}
\lim_{k\to\infty} \left| \nu_{n_k} \right|(W) = \infty \,.
\end{displaymath}
In particular, for each precompact set $K$ with non-empty interior we must have
\begin{displaymath}
\lim_{k\to\infty} \| \nu_{n_k} \|_K = \infty \,.
\end{displaymath}

One natural question is what happens when we move away from translation bounded measures. If $\mu_n, \mu$ are Fourier transformable measures such that $\mu_n \to \mu$, is it possible for $\widehat{\mu_n}$ to have other cluster points than $\widehat{\mu}$?
By replacing the sequence by the subsequence $(\widehat{\mu_{n_k}})_{k\in\N}$ converging to one such different cluster point, and then defining $\nu_k:= \mu_{n_k}-\mu$, the question becomes equivalent to the following question.
\begin{question}
Does there exist a sequence $(\nu_n)_{n\in\N}$ of Fourier transformable measures and some measure $\nu \neq 0$ such that, in the vague topology, we have
\begin{align*}
  \lim_{n\to\infty} \nu_n & =0 \\
  \lim_{n\to\infty} \widehat{\nu_n} &= \mu \neq 0 \,.
\end{align*}
\end{question}

As discussed above, if such a sequence exists, then we must have $\lim_{n\to\infty} \| \widehat{\nu_n} \|_K =\infty$, by Corollary~\ref{T2}. Moreover, Theorem~\ref{thm:mainvag2} implies that, for each $C>0$, there exists some $N$ such that, for all $n>N$ we have $\| \nu_n \|_K  >C $. We could potentially have $\| \nu_n \|_K = \infty$.

\medskip

Recall that the mapping $\ \widehat{}\, : (\cM^\infty(\R^d), \| \cdot \|_K) \to \cS'$ is continuous. A natural question is then the following.

\begin{question}\label{Q2}
Is $\ \widehat{}\, : (\cM^\infty_T(\R^d), \| \cdot \|_K) \to (\cM^\infty(\R^d), \mbox{ vague topology})$ continuous?
\end{question}

We give below an equivalent formulation for Question~\ref{Q2}.

\begin{proposition}
Fix some compact set $K_0$ with non-empty interior. Then, the following statements are equivalent:
\begin{itemize}
\item [(i)] $\widehat{}\, : (\cM^\infty_T(\R^d), \| \cdot \|_{K_0}) \to (\cM^\infty(\R^d), \mbox{ vague topology})$ is continuous.
\item [(ii)] For each pair $K,W \subseteq \R^d$ of compact sets with non-empty interior there exists some $C=C(K,W)$ such that
\begin{displaymath}
\left| \widehat{\mu} \right|(W) \leqslant C \| \mu \|_K \qquad \text{ for all }\mu \in \cM^\infty_T(\R^d) \,.
\end{displaymath}
\item [(iii)] There exists some $C$ such that
\begin{displaymath}
| \widehat{\mu}|([0,1]^d) \leqslant C \| \mu \|_{[0,1]^d} \qquad \text{ for all }\mu \in \cM^\infty_T(\R^d) \,.
\end{displaymath}
\end{itemize}
\end{proposition}
\begin{proof} (i)$\implies$(iii)
Assume by contradiction that this is not true. Then, for each $n \in \N$, there exists some $\nu_n \in \cM^\infty_T(\R^d)$ such that
\begin{displaymath}
\left| \widehat{\nu_n} \right|([0,1]^d) > n\, \| \nu_n \|_{[0,1]^d} \,.
\end{displaymath}
Note here that $\widehat{\nu_n} \neq 0$ and hence $\| \nu_n \|_{[0,1]^d} \neq 0$.

Define
\begin{displaymath}
\mu_n:= \frac{1}{\sqrt{n}\,  \| \nu_n \|_{[0,1]^d}}\,  \nu_n
\end{displaymath}
Then, $\| \mu_n \|_{[0,1]^d} =\frac{1}{\sqrt{n}}$ and hence $\mu_n \to 0$ in $(\cM^\infty_T(\R^d), \| \cdot \|_{K_0})$, since $[0,1]^d$ and $K_0$ define equivalent norms \cite{BL,SpSt}.
By (i) it follows that $\widehat{\mu_n}$ is vaguely convergent to 0, and hence vaguely bounded. Therefore, by Proposition~\ref{P2} the set $\{ \left| \widehat{\mu_n} \right|([0,1]^d)\, :\, n\in\N \}$ is bounded. But this is not possible, as
\begin{displaymath}
\left| \widehat{\mu_n} \right|([0,1]^d)=\frac{1}{\sqrt{n}\,  \| \nu_n \|_{[0,1]^d} }\, \left| \widehat{\nu_n} \right|([0,1]^d) > \frac{1}{\sqrt{n}\,  \| \nu_n \|_{[0,1]^d}}\, n\, \| \nu_n \|_{[0,1]^d} =\sqrt{n} \,.
\end{displaymath}

\noindent (iii)$\implies$(ii) Follows immediately from the fact that both $K,W$ can be covered by finitely many translates of $[0,1]^d$.

\noindent (ii)$\implies$(i) Since the Fourier transform is a linear operator, it suffices to show continuity at $0$. Let $(\mu_n)_{n\in\N}$ be an arbitrary sequence which converges to $0$ in $\| \cdot \|_{K_0}$. We need to show that $\widehat{\mu_n} \to 0$.

Let $W \subseteq \R^d$ be an arbitrary compact set. By (ii) the sequence $\left| \widehat{\mu_n} \right|(W)\to0$ and hence, it is bounded. Proposition~\ref{P2} then implies that $\{ \widehat{\mu_n}\, :\,n\in \N \}$ is vaguely pre-compact.
It follows that
\begin{itemize}
\item{} $\mu_n \to 0$ in the vague topology.
\item{} $\{ \mu_n : n \in \N \}$ are equi translation bounded.
\item{}  $\{ \widehat{\mu_n}\, :\, n\in\N \}$ is vaguely pre-compact.
\end{itemize}
Theorem~\ref{T3} implies that $\widehat{\mu_n} \to 0$ as claimed.
\end{proof}

\appendix
\section{The vague topology on arbitrary LCAG}\label{sect:app}

Recall that given a Banach space $(B, \| \cdot \|)$ its dual space $B^*$ becomes a Banach space with the norm
\begin{displaymath}
\| f \|:= \sup \{ |f(x)|\, :\, x \in B,\, \|x \| \leqslant 1 \}
\end{displaymath}
Let us first note the following well known result on the weak*-topology.

\begin{proposition}\label{P1}
Let $B$ be a Banach space and let $A \subseteq B^*$ be any set. Then, the following statements are equivalent:
\begin{itemize}
\item [(i)] The weak*-closure of $A$ is weak*-compact.
\item [(ii)] $A$ is weak*-bounded, that is, for each $x \in B$, the set $\{ f(x) : f \in A \}$ is bounded.
\item [(iii)] $A$ is norm bounded.
\end{itemize}
\end{proposition}
\begin{proof}
(ii)$\iff$(iii) This is a direct consequence of the uniform bounded principle.

\smallskip

\noindent (iii)$\implies$(i) This follows from the Banach Alaoglu theorem.

\smallskip

\noindent (i)$\implies$(ii) This is the standard ``compact implies bounded" argument.
\end{proof}

Recall that, for each fixed compact set $K$, the set
\begin{displaymath}
C(G:K):= \{ f \in \Cc(G)\, :\, \supp(f) \subseteq K \}
\end{displaymath}
is a Banach space. By the Riesz representation theorem, the dual space of this space can be identified with the space $\cM(K)$ of Radon measures supported on $K$.

The following is a simple computation (compare \cite{SpSt}).

\begin{lemma}
The dual norm of the duality $ (\Cc(G:K), \| \cdot \|_\infty)^* = \cM(K)$ is given by
\begin{displaymath}
\| \mu \|:= \left| \mu \right|(K) \,.
\end{displaymath}
\end{lemma}

Let us now extend the following definition to arbitrary LCAG.

\begin{definition} A set $\mathcal{A} \subseteq \cM(G)$ is called \emph{vaguely bounded} if, for each $f \in \Cc(G)$, the set $\{ \mu(f): \mu \in \mathcal{A} \}$ is bounded.
\end{definition}

We can now prove the following result, compare \cite{BL}.

\begin{proposition}\label{P2A}
Let $\mathcal{A} \subseteq \cM(G)$. Then, the following statements are equivalent:
\begin{itemize}
\item [(i)] $\mathcal{A}$ has compact vague closure.
\item [(ii)] $\mathcal{A}$ is vaguely bounded.
\item [(iii)] For each compact set $K \subseteq G$, the set $\{ \left| \mu \right| (K) : \mu  \in \mathcal{A} \}$ is bounded.
\item[(iv)] There exists a collection $\{ K_\alpha\, :\,  \alpha\}$ of compact sets in $G$ such that
\begin{itemize}
\item[$\bullet$] $G =\bigcup_{\alpha} (K_\alpha)^\circ$ and
\item[$\bullet$] $\{ |\mu| (K_\alpha) : \mu  \in \mathcal{A} \}$ is bounded, for each $\alpha$.
\end{itemize}
\end{itemize}
Moreover, if $G$ is second countable and the above hold, the vague topology is metrisable on $\cA$.
\end{proposition}
\begin{proof}
(i)$\implies$(ii) Let $f \in \Cc(G)$ be arbitrary. The function $F : \cM(G) \to \CC$ defined by
\begin{displaymath}
F(\mu):= \mu(f)
\end{displaymath}
is continuous. Since $\overline{\cA}$ is compact, $F(\overline{\cA})$ is compact in $\CC$ and hence bounded. The claim follows.

\smallskip

\noindent (ii)$\implies$(iii) Let $K\subseteq G$ be compact and let $W$ be a compact set such that $K \subseteq W^\circ$.
Each $\mu$ defines an operator $\mu : \Cc(G:W) \to \CC$. By (ii), for each $f \in \Cc(G:W)$, the set $\{ \mu(f) : \mu \in \cA \}$ is bounded. Therefore, by the uniform bounded principle, there exists a constant $C>0$ such that, for all $f \in \Cc(G:W)$ with $\| f \|_\infty =1$ and for all $\mu \in \cA$, we have
\begin{displaymath}
\left| \mu (f) \right| \leqslant C \,.
\end{displaymath}
Therefore, for all $f \in \Cc(G:W)$ we have
\begin{displaymath}
\left| \mu (f) \right| \leqslant C \| f \|_\infty \,.
\end{displaymath}

Now, let $\mu \in \cA$. Since $K \subseteq W^\circ$, we can pick some $f \in \Cc(G)$ with $1_K \leqslant f \leqslant 1_W$. By definition of $\left| \mu \right|$, there exists some $g \in \Cc(G)$ with $|g| \leqslant f$ such that
\begin{displaymath}
\left| \mu \right| (f) \leqslant \left| \mu (g) \right| +1 \,.
\end{displaymath}
In particular, we have $g \in \Cc(G:W)$ and $\| g \|_\infty <1$. Therefore,
\begin{displaymath}
| \mu | (K) \leqslant | \mu | (f) \leqslant \left| \mu (g) \right| +1  \leqslant C \| g \|_\infty +1=C+1 \,.
\end{displaymath}
This shows that
\begin{displaymath}
\left| \mu \right| (K) \leqslant C+1 \qquad \text{ for all } \mu \in \cA \,.
\end{displaymath}

\smallskip

\noindent (iii)$\implies$(iv) Trivial.

\medskip

\noindent (iv)$\implies$(i) Define the mapping
\begin{displaymath}
j : \cM(G) \hookrightarrow \prod_{\alpha} \cM(K_\alpha) \,,\qquad j(\mu) = \left( \mu|_{K_\alpha} \right)_\alpha \,.
\end{displaymath}
It is clear that this mapping is continuous from the vague topology to the product topology, where each $\cM(K_\alpha)$ is equipped with the vague topology. Moreover, since $G =\bigcup_{\alpha} K_\alpha^\circ$, it follows immediately that $j$ is a homeomorphism onto its image.

For each $\alpha$ define
\begin{displaymath}
C_\alpha :=\sup \{ |\mu| (K_\alpha) : \mu  \in \mathcal{A} \}
\end{displaymath}
By Proposition~\ref{P1}, the set
\begin{displaymath}
M_\alpha:= \{ \nu \in \cM(K_\alpha) : \left| \nu \right|(K_\alpha) \leqslant C_\alpha \}
\end{displaymath}
is compact, and hence, so is $\prod_{\alpha} M_\alpha$.
Therefore, the set $\prod_{\alpha} M_\alpha \cap j( \cM(G) )$ has compact closure in $\prod_{\alpha} \cM(K_\alpha)$ and hence the set $j^{-1} (\prod_{\alpha} M_\alpha \cap j( \cM(G) ))= j^{-1}(\prod_{\alpha} M_\alpha)$ has compact closure in $\cM(G)$.
Now, by the definition of $C_\alpha$, we have $\cA \subseteq j^{-1}(\prod_{\alpha} M_\alpha)$. The claim follows.

\smallskip

Finally, if $G$ is second countable, we can find a sequence $(K_n)_{n\in\N}$ of compact sets such that $K_n \subseteq (K_{n+1})^\circ$ and $G= \bigcup_{n\in\N} K_n$. Note that in this case we have $G= \bigcup_{n\in\N} K_n \subseteq \bigcup_{n\in\N} K_{n+1}^\circ$ and hence $G= \bigcup_{n\in\N} K_n$.

Define as above
\begin{displaymath}
j : \cM(G) \hookrightarrow \prod_{n\in\N} \cM(K_n) \,, \qquad j(\mu) = \left( \mu|_{K_n} \right)_{n\in\N} \,.
\end{displaymath}
Let
\begin{displaymath}
C_n:=\sup \{ |\mu| (K_n) : \mu  \in \mathcal{A} \}
\end{displaymath}

Next, by the second countability of $G$, the space $C(G:K_n):= \{ f \in \Cc(G): \supp(f) \subseteq K_n \}$ is a separable Banach space. Therefore, by \cite[Thm.~3.16]{Rud2}, the weak-*-topology is metrisable on
\begin{displaymath}
M_n:= \{ \nu \in \cM(K_n) : \left| \nu \right|(K_n) \leqslant C_n \} \,,
\end{displaymath}
and hence, $\prod_{n\in\N} M_n$. As above, $j$ is a homeomorphism on its image, and hence $j^{-1}( \prod_{n\in\N} M_n)$ is metrisable.
\end{proof}

Let us note the following consequence, compare \cite[Thm.~2]{BL}.

\begin{proposition}\label{P3A}
Let $(U_\alpha)_{\alpha}$ be a collection of open precompact sets such that $G= \bigcup_\alpha U_\alpha$. Let $C_\alpha >0$ be constants, for all $\alpha$. Then, the space
\begin{displaymath}
\MM:= \{ \mu \in \cM(G)\, :\, |\mu| (U_\alpha) \leqslant C_\alpha \text{ for all } \alpha \}
\end{displaymath}
is vaguely compact. If $G$ is second countable, the vague topology is metrisable on $\MM$.
\end{proposition}
\begin{proof}
By Proposition~\ref{P2A}, applied with $K_\alpha = \overline{U_\alpha}$, the set $\MM$ has compact closure in $\cM(G)$. We show that $\MM$ is closed.

Let $(\mu_\beta)_{\beta}$ be a net in $\MM$ which converges to some $\mu \in \cM(G)$. We need to show that $\mu \in \MM$.
Fix some $\alpha$ and $\eps >0$.
Pick some $f \in \Cc(G)$ such that $0 \leqslant f \leqslant 1_{U_\alpha}$ and
\begin{displaymath}
|\mu| (f) \geqslant  |\mu| (U_\alpha) - \frac{\eps}{3} \,.
\end{displaymath}
Next, by the definition of $|\mu|$, there exists some $g \in \Cc(G)$ with $|g| \leqslant f$ and
\begin{displaymath}
|\mu (g)| \geqslant  |\mu| (f) - \frac{\eps}{3} \,.
\end{displaymath}
In particular, $g$ is zero outside $U_\alpha$.
Finally, by the vague convergence of $(\mu_\beta)_{\beta}$ to $\mu$, there exists some $\beta$ such that
\begin{displaymath}
\left| \mu_\beta(g)- \mu(g) \right| \leqslant \frac{\eps}{3} \,.
\end{displaymath}
Therefore, we have
\begin{align*}
|\mu| (U_\alpha) & \leqslant |\mu| (f) +\frac{\eps}{3} \leqslant |\mu (g)| +\frac{2\eps}{3}   \\
  &\leqslant  |\mu_\beta (g)| + \eps \leqslant  \left| \int_G g(x)\ \dd
    \mu_\beta(x) \right| +\eps = \left| \int_{U_\alpha} g(x)\ \dd
    \mu_\beta(x) \right| +\eps \\
  &\leqslant \| g \|_\infty |\mu_\beta| (U_\alpha) +\eps \leqslant 1 \cdot
     C_\alpha +\eps = C_\alpha +\eps \,.
\end{align*}
This shows that, for all $\eps >0$, we have
\begin{displaymath}
|\mu| (U_\alpha) \leqslant C_\alpha +\eps \,.
\end{displaymath}
The claim follows.
\end{proof}

\begin{remark}
It follows from Proposition~\ref{P2} that a set $\mathcal{A}$ is vaguely precompact if and only if it is a subset of some $\MM$ as in Proposition~\ref{P3}.
\end{remark}

\begin{coro}\label{cor BL thm2} \cite[Thm.~2]{BL}
For each open precompact $U \subseteq G$ and each $C>0$, the space
\begin{displaymath}
\cM_{C,U}:= \{ \mu\in\cM(G)\, :\, \| \mu \|_K \leqslant C \}
\end{displaymath}
is vaguely compact. If $G$ is second countable, then the vague topology is metrisable on $\cM_{C,U}$.
\end{coro}

As an interesting consequence, we get the following result, which was first observed by Schlottmann \cite[page~145]{Martin2}.

\begin{theorem}\label{T4}
Let $\Omega \subseteq \cM(G)$ be any closed $G$-invariant set. Then, $\Omega$ is vaguely compact if and only if $\Omega$ is equi translation bounded.
\end{theorem}
\begin{proof}
$\Longrightarrow$: Let $K$ be any compact set. By Proposition~\ref{P2A}, the set $\{ \left| \mu \right| (K) : \mu  \in \Omega \}$ is bounded, let $C>0$ be any upper bound for this. Since for all $t \in G$ and $\mu \in \Omega$ we have $T_t \mu \in \Omega$ and hence
\begin{displaymath}
\left| \mu \right| (t+K)=\left| T_t\mu \right| (K) \leqslant C \,.
\end{displaymath}
It follows that, for all $\mu \in \Omega$, we have
\begin{displaymath}
\| \mu \|_K = \sup_{t \in G} \left| \mu \right| (t+K) \leqslant C \,.
\end{displaymath}

\noindent $\Longleftarrow$: Let $K \subseteq G$ be arbitrary. Since $\Omega$ is equi translation bounded, there exists a constant $C>0$ such that $\| \mu \|_K  \leqslant C$ for all $\mu \in \Omega$.
In particular, the set $\{ \left| \mu \right| (K) : \mu  \in \Omega \}$ is bounded by $C$, and hence $\Omega$ is vaguely pre-compact. Since $\Omega$ is closed, it is compact.
\end{proof}

\begin{remark}
In \cite{BL}, the authors define a dynamical system on the translation bounded measures on G (TMDS) to be a pair $(\Omega, G)$ where $\Omega \subseteq \cM(G)$ is any closed $G$-invariant set of measures which are equi translation bounded.
They observe that any such set must be compact, and hence a topological dynamical system.

Theorem~\ref{T4} above says that the converse is true, whenever we have a topological dynamical system $(\Omega, G)$, where $\Omega \subseteq \cM(G)$ is equipped with the vague topology, then it is a TMDS.
\end{remark}

\smallskip

An immediate consequence of Theorem~\ref{T4} is the following.

\begin{coro}
Let $\mu \in \cM(G)$, and let $\XX(\mu) =\overline{ \{ T_t \mu :t \in G \}}$ be the hull of this measure, with the closure being taken in the vague topology. Then, $\XX(\mu)$ is compact if and only if $\mu$ is translation bounded.

Moreover, if $G$ is $\sigma$-compact and $\mu \in \cM^\infty(G)$, then the vague topology is metrisable on $\XX(\mu)$.
\end{coro}

\bigskip

We complete the paper by looking at some examples which emphasize the importance of second countability of $G$ for diffraction theory. These show that, outside second countable groups, one should work with van Hove nets and not sequences.

Let us start with an example of a group which is compact (hence $\sigma$-compact) but not metrisable, and an equi-translation bounded sequence of measures which has no convergent subsequence.

\begin{example}\label{ex5} Consider $G:= (\R/\Z)^{\R/\Z}$. Then $G$ is a group, which is compact with respect to the product topology. Similarly to $I^I$ , one can show that $G$ is not sequentially compact (compare \cite{SeeSte}). Therefore, there exists a sequence $(x_n)_{n\in\N}$ with no convergent subsequence.

Define $\mu_n:= \delta_{x_n}$, for all $n\in\N$. Then, $(\mu_n)_{n\in\N}$ is equi-bounded and hence equi-translation bounded. We show that this sequence has no convergent subsequence.

Assume by contradiction that there exists a subsequence $(\mu_{n_k})_{k\in\N}$ convergent to some measure $\mu$.
Since $G$ is compact, there exists a directed set $I$ and a monotone final function $h: I \to \N$, such that the subnet $(x_{n_{h(\beta)}})_{\beta
\in I }$ is convergent to some $a \in G$.
Now, since $a =\lim_\beta x_{k_{n(\beta)}}$, we have 
\begin{displaymath}
f(a)=\lim_\beta f(x_{n_{h(\beta)}})
\end{displaymath}
for all $f \in \Cc(G)$ by the continuity of $f$.

Fix an arbitrary $f \in \Cc(G)$. Now, since $(\mu_{n_k})_{k\in\N}$ converges vaguely to $\mu$, any subnet of $(\mu_{n_k}(f))_{k\in\N}$ converges to $\mu(f)$. In particular,
\begin{displaymath}
\lim_\beta \mu_{n_{h(\beta)}}(f)=\mu(f) \,.
\end{displaymath}
Thus, we have
\begin{displaymath}
\mu(f)=\lim_\beta \mu_{n_{h(\beta)}}(f)=\lim_\beta f(x_{n_{h(\beta)}})=f(a) \,.
\end{displaymath}
Therefore, $\mu=\delta_a$ and hence
\begin{displaymath}
\delta_a= \lim_{k\to\infty} \delta_{x_{n_k}} \,.
\end{displaymath}

Next, since $(x_{n_k})_{k\in\N}$ does not converges to $a$, there exists an open set $U\ni a$ and a subsequence $(x_{n_{k_{\ell}}})_{\ell\in\N}$ of $(x_{n_k})_{k\in\N}$ such that, for all $n\in \N$, we have $x_{n_{k_{\ell}}} \notin U$. By Urysohn's lemma, there exists some $f \in \Cc(G)$ such that $f(a)=1$ and $f(x)=0$ for all $x \notin U$. Since $\delta_a= \lim_{k\to\infty} \delta_{x_{n_k}}$ and $(x_{n_{k_{\ell}}})_{\ell\in\N}$ is a subsequence of $(x_{n_k})_{k\in\N}$, we have $\delta_a= \lim_{\ell\to\infty} \delta_{x_{n_{k_{\ell}}}}$ and hence
\begin{displaymath}
1=f(a)=\lim_{\ell\to\infty} f(x_{n_{k_{\ell}}})=\lim_{\ell\to\infty} 0=0 \,.
\end{displaymath}
Since we obtained a contradiction, our assumption is wrong. Therefore, $(\mu_n)_{n\in\N}$ has no convergent subsequence.
\end{example}

\section{Second countability and the existence of the autocorrelation}

In this section, we discuss the necessity of second countability of $G$ to ensure the existence of the autocorrelation for a translation bounded measure $\mu$.

\smallskip

Let us start with the definition of a F\o lner and van Hove sequence.

\begin{definition} 
A sequence $(A_n)_{n\in\N}$ of precompact subsets of $G$ is called a F\o lner sequence if, for each $x \in G$, we have
\begin{displaymath}
\lim_{n\to\infty} \frac{|F_n \Delta (x+F_n)|}{|F_n|} =0 \,.
\end{displaymath}

A sequence $(A_n)_{n\in\N}$ of precompact subsets of $G$ is called a \textbf{van Hove sequence} if, for each compact set $K \subseteq G$, we have
\[
\lim_{n\to\infty} \frac{|\partial^{K} A_{n}|}{|A_{n}|}  =  0 \, ,
\]
where the \textbf{$K$-boundary $\partial^{K} A$} of an open set $A$ is defined as
\[
\partial^{K} A := \bigl( \overline{A+K}\setminus A\bigr) \cup
\bigl((\left(G \backslash A\right) - K)\cap \overline{A}\, \bigr) \,.
\]
\end{definition}

If $G$ is $\sigma$-compact, then van Hove sequences exists in $G$ \cite{Martin2}. It is well known that each van Hove sequence is a F\o lner sequence.

\smallskip

Now, let us look at the (classical) definition of the autocorrelation. First, we need the following simple result.

\begin{lemma}\label{lemma aut exists} 
Let $G$ be a second countable LCAG, $\omega \in\cM^\infty(G)$ and $( A_n)_{n\in\N}$ a van Hove sequence in $G$. Define
\begin{displaymath}
\gamma_n:= \frac{\omega_n *\widetilde{\omega_n}}{|A_n|}  \,,
\end{displaymath}
where $\omega_n:=\omega|_{A_n}$ is the restriction of $\omega$ to $A_n$.
Then, $\{ \gamma _n\, :\, n\in\N \}$ is vaguely precompact. In particular, there exists a subsequence $(\gamma_{n_k})_{k\in\N}$ which converges to some $\gamma \in \cM^\infty(G)$.
\end{lemma}
\begin{proof}
First, by \cite[Lem.~1.1(b)]{Martin2}, the sequence $\frac{1}{|A_n|} \left| \widetilde{\omega_n} \right|(G)$ is bounded. Let $C>0$ be a constant such that
\begin{displaymath}
\frac{1}{|A_n|} \left| \widetilde{\omega_n} \right|(G) \leqslant C \qquad \text{for all } n\in\N \,.
\end{displaymath}
Next, fix any open precompact set $U$ and pick a compact set $W$ such that $\overline{U} \subseteq W^\circ$. By the translation boundedness of $\omega$, we have
\begin{align*}
\| \omega_n \|_W&= \sup_{x \in G} \{ \left| \omega_n \right|(x+W) \} = \sup_{x \in G} \{ \left| \omega \right|\left((x+W)\cap A_n \right) \}\\
&\leqslant \sup_{x \in G} \{ \left| \omega \right|(x+W) \}= \| \omega \|_W < \infty \,.
\end{align*}
Then, by \cite[Lem.~6.1]{NS13}, we have for all $x \in G$
\begin{displaymath}
\left| \frac{1}{|A_n|} \omega_n *\widetilde{\omega_n} \right|(x+U) \leqslant \| \omega_n \|_{x+W} \cdot \left| \frac{1}{|A_n|} \widetilde{\omega_n} \right|(G) \,.
\end{displaymath}
Since $ \| \omega_n \|_{x+W} = \| \omega_n \|_{W}$ we have
\begin{align*}
  \| \gamma_n \|_U &= \sup_{x \in G} \{ \Big| \frac{1}{|A_n|} \omega_n *\widetilde{\omega_n} \Big|(x+K) \}  \\
  &\leqslant \sup_{x \in G} \{ \| \omega_n \|_{W} \cdot \Big| \frac{1}{|A_n|} \widetilde{\omega_n} \Big|(G) \} \\
  &\leqslant C \cdot \| \omega \|_W =: C'\,.
\end{align*}
Therefore, for all $n\in\N$, we have
\begin{displaymath}
\gamma_n \in \cM_{C',U}=\{ \nu \in \cM^\infty(G) : \| \nu \|_U \leqslant C '\} \,.
\end{displaymath}

Since this space is vaguely compact and metrisable by \cite[Thm.~2]{BL} or Corollary~\ref{cor BL thm2}, the claim follows.
\end{proof}

\smallskip

\begin{definition}\label{def autoc} Let $G$ be a second countable LCAG, $(A_n)_{n\in\N}$ be any van Hove sequence in $G$. If $\omega$ is any translation bounded measure, then any limit $\gamma$ of a subsequence of the sequence $(\gamma_n)_{n\in\N}$ defined in Lemma~\ref{lemma aut exists} is called an \emph{autocorrelation of $\omega$}.
\end{definition}

Lemma~\ref{lemma aut exists} says that in second countable LCAG every translation bounded measure has an autocorrelation.

\smallskip
We next show that second countability of $G$ is essential for the existence of the autocorrelation. Recall first that for a LCAG second countability is equivalent to $\sigma$-compactness and metrisability (see for example \cite[Thm.~2.B.2 and Thm.~2.B.4]{CoHa}).

\bigskip

First we show that  $\sigma$-compactness is necessary for the existence of van Hove sequences. In particular, $\sigma$ compactness will be a necessary condition for the existence of the autocorrelation of a single object.

Let us start with the following simple lemma.

\begin{lemma}\label{lem:folner} 
Let $(F_n)_{n\in\N}$ be a F\o lner sequence in $G$. Then,
\begin{displaymath}
G= \bigcup_{n\in\N} (F_n-F_n) \,.
\end{displaymath}
In particular, any group admiting a F\o lner sequence is $\sigma$-compact.
\end{lemma}
\begin{proof}
Assume by contradiction that $G \neq \bigcup_{n\in\N} (F_n-F_n)$. Let $y \in G \backslash \bigcup_{n\in\N} (F_n-F_n)$. Then, for each $n$, we have $y \notin F_n-F_n$ and hence $F_n \cap (y+F_n) =\varnothing$. This gives 
\[
|F_n \Delta (y+F_n)|= |F_n|+ |y+F_n| =2 |F_n|\,. 
\]
Therefore,
\begin{displaymath}
2=\frac{|F_n \Delta (x+F_n)|}{|F_n|} \qquad \text{for all } n\in\N \,,
\end{displaymath}
which contradicts the definition of F\o lner sequence.

The last claim follows from
\begin{displaymath}
G= \bigcup_{n\in\N} (F_n-F_n) \subseteq \bigcup_{n\in\N} (\overline{F_n}-\overline{F_n}) \,.
\end{displaymath}
\end{proof}

\begin{remark} 
The commutativity of $G$ is not important in Lemma~\ref{lem:folner}. Indeed, if $(F_n)_{n\in\N}$ is a F\o lner sequence in an arbitrary LCG, then exactly as in Lemma~\ref{lem:folner} one can show that
\begin{displaymath}
G= \bigcup_{n\in\N} F_n^{-1}F_n= \bigcup_{n\in\N} F_nF_n^{-1} \,.
\end{displaymath}
\end{remark}

As an immediate consequence we get the following result.

\begin{proposition} 
Let $G$ be a LCAG. Then, the following statements are equivalent.
\begin{itemize}
  \item [(i)] $G$ is $\sigma$-compact.
  \item [(ii)] There exists a van Hove sequence in $G$.
  \item [(iii)] There exists a F\o lner sequence in $G$.
\end{itemize}
\end{proposition}
\begin{proof}
(i) $\Longrightarrow$ (ii) follows from \cite{Martin2}.

(ii) $\Longrightarrow$ (iii) is obvious.

(iii) $\Longrightarrow$ (i) follows from Lemma~\ref{lem:folner}.
\end{proof}

Since the (usual) definition of the autocorrelation measure requires the existence of a van Hove sequence, the autocorrelation can only exist in $\sigma$-compact groups.

\bigskip

Next, we discuss the issue of metrisability of $G$. We show here that there exist translation bounded measures on $\sigma$-compact LCAG $G$ which do not have an autocorrelation as defined in Definition~\ref{def autoc}.

\begin{example} 
Let $\KK:= (\R/\Z)^{\R/\Z}$ and let $G:= \KK \times \Z$, with the Haar measure being the product between the probability Haar measure on $\KK$ and the counting measure on $\Z$. Then, $G$ is a $\sigma$-compact group.

Let $(x_n)_{n\in\N}$ be a sequence in $\KK$ as in Example~\ref{ex5}. Define
\begin{displaymath}
\omega:= \sum_{m=1}^\infty \sum_{k=(m-1)!+1}^{m!} \left(\delta_{(0,k)}+ \delta_{(x_m,k)} \right) \,.
\end{displaymath}
This is a translation bounded measure. Let $A_n:= \KK \times \bigl( (-n!,n!) \cap \Z \bigr)$, for all $n\in\N$. Then, $(A_n)_{n\in\N}$ is a van Hove sequence in $G$.

We claim that the sequence $(\gamma_n)_{n\in\N}$ has no convergent subsequence. Note that by the translation boundedness of $\omega$, the set $\{ \gamma_n\, :\, n\in\N \}$ has compact closure (compare Lemma~\ref{lemma aut exists}).
\end{example}
\begin{proof}
A simple computation yields
\begin{align*}
\gamma_n & = \frac{1}{|A_n|} \omega_n *\widetilde{\omega_n}\\
  &= \frac{1}{2n!+1} \left( \sum_{m=1}^{n} \sum_{k=(m-1)!+1}^{m!}
     \left(\delta_{(0,k)}+ \delta_{(x_m,k)} \right) \right)*\left( 
     \sum_{j=1}^{n} \sum_{l=(j-1)!+1}^{m!} \left(\delta_{(0,-l)}+ 
     \delta_{(-x_j,-l)} \right) \right) \\
  &= \frac{1}{2n!+1} \left( \sum_{m,j=1}^{n} \sum_{k=(m-1)!+1}^{m!} 
      \sum_{l=(j-1)!+1}^{j!}
     \left(\delta_{(0,k)}+ \delta_{(x_m,k)} \right)* \left(\delta_{(0,-l)}
     + \delta_{(-x_j,-l)} \right) \right) \\
 & =\frac{1}{2n!+1} \left( \sum_{m,j=1}^{n} \sum_{k=(m-1)!+1}^{m!} 
     \sum_{l=(j-1)!+1}^{j!} \left(\delta_{(0,k-l)}+ \delta_{(x_m,k-l)}
     + \delta_{(-x_j,k-l)}+ \delta_{(x_m-x_j,k-l)} \right) \right) \,.
\end{align*}
Define
\begin{displaymath}
\nu_n:=   \gamma_n|_{\KK \times \{ 0 \}} \,.
\end{displaymath}
In this situation, we have the conditions
\[
k=l\,, \qquad (m-1)!+1\leqslant k \leqslant m! \qquad \text{ and } \qquad (j-1)!+1 \leqslant l \leqslant j! \,,
\]
which give $m!\leqslant j!$ and $j!\leqslant m!$. Hence, we have $j=m$.
By the above, we have
\begin{align*}
  \nu_n &=\frac{1}{2n!+1} \left( \sum_{m=1}^{n} \sum_{k=(m-1)!+1}^{m!} \left(\delta_{(0,0)}+ \delta_{(x_m,0)} + \delta_{(-x_m,0)}+ \delta_{(0,0)} \right) \right) \\
  &=\frac{1}{2n!+1} \left( \sum_{m=1}^{n} \left( m!-(m-1)!\right) \left( 2\delta_{(0,0)}+ \delta_{(x_m,0)} + \delta_{(-x_m,0)} \right) \right) \\
  &=\left( \frac{2(n!-1)}{2n!+1} \delta_{(0,0)} \right)+  \left( \sum_{m=1}^{n} \frac{m!-(m-1)!}{2n!+1} \left(  \delta_{(x_m,0)} + \delta_{(-x_m,0)} \right) \right) \,.
\end{align*}

Assume by contradiction that $(\gamma_n)_{n\in\N}$ has a subsequence $(\gamma_{n_k})_{k\in\N}$ which converges vaguely. Since $\KK \times \{ 0 \}$ is closed and open in $G$, the restriction $(\nu_{n_k})_{k\in\N}$ with $\nu_{n_k}=\gamma_{n_k}|_{\KK \times \{ 0 \}}$ is vaguely convergent. Moreover, since $(\frac{2((n_k)!-1)}{2(n_k)!+1} \delta_{(0,0)})_{k\in\N}$ converges vaguely to $\delta_{(0,0)}$, the sequence $(\nu_{n_k}-\frac{2((n_k)!-1)}{2(n_k)!+1} \delta_{(0,0)})_{k\in\N}$ is vaguely convergent. Let
\begin{displaymath}
\mu_k:= \nu_{n_k}-\frac{2((n_k)!-1)}{2(n_k)!+1} \delta_{(0,0)}=  \sum_{m=1}^{n_k} \frac{m!-(m-1)!}{2k!+1} \left(  \delta_{(x_m,0)} + \delta_{(-x_m,0)} \right)  \,.
\end{displaymath}
Then, by the above, $(\mu_k)_{k\in\N}$ is vaguely convergent.

Note next that
\begin{align*}
\Big\|  \sum_{m=1}^{n_k-1} \frac{m!-(m-1)!}{2(n_k)!+1} \left(  \delta_{(x_m,0)} 
    + \delta_{(-x_m,0)} \right) \Big\|  
    &= \left| \sum_{m=1}^{n_k-1} \frac{m!-(m-1)!}{2(n_k)!+1} \left(
      \delta_{(x_m,0)} + \delta_{(-x_m,0)} \right) \right|(G) \\
    &\leqslant \sum_{m=1}^{n_k-1}  \left| \frac{m!-(m-1)!}{2(n_k)!+1} \left( 
      \delta_{(x_m,0)} + \delta_{(-x_m,0)} \right) \right|(G)\\
    &=\sum_{m=1}^{n_k-1} 2 \frac{m!-(m-1)!}{2(k_n)!+1}\\
    &= \frac{2 (n_k-1)!-2}{2 (n_k)!+1}  \,.
\end{align*}
As $\lim_{n\to\infty} \frac{2 (n-1)!-2}{2 n!+1}=0$, it follows that $\lim_{k\to\infty} \frac{2 (n_k-1)!-2}{2 (n_k)!+1}=0$ and hence 
\[
\sum_{m=1}^{n_k-1} \frac{m!-(m-1)!}{2(n_k)!+1} \left(  \delta_{(x_m,0)} + \delta_{(-x_m,0)} \right)
\] 
converges vaguely to zero.
Therefore,
\begin{displaymath}
\mu_k-\sum_{m=1}^{n_k-1} \frac{m!-(m-1)!}{2(n_k)!+1} \left(  \delta_{(x_m,0)} + \delta_{(-x_m,0)} \right)=\frac{(n_k)!-(n_k-1)!}{2(n_k)!+1} \left(  \delta_{(x_{n_k},0)} + \delta_{(-x_{n_k},0)} \right)
\end{displaymath}
is vaguely convergent. Similarly to the above, $\lim_{k\to\infty}\frac{(n_k-1)!}{2(n_k)!+1} \left(  \delta_{(x_{n_k},0)} + \delta_{(-x_{n_k},0)} \right)=0$ in the vague topology, and hence $\frac{(n_k)!}{2(n_k)!+1} \left(  \delta_{(x_{n_k},0)} + \delta_{(-x_{n_k},0)} \right)$ is vaguely convergent.

Finally, since $\frac{2(n_k)!+1}{(n_k)!}$ converges to $2$, we get that $(\varphi_{k})_{k\in\N}$ with
\begin{displaymath}
\varphi_k:= \delta_{(x_{n_k},0)} + \delta_{(-x_{n_k},0)}= \frac{2(n_k)!+1}{(n_k)!}\left(\frac{(n_k)!}{2(n_k)!+1} \left(  \delta_{(x_{n_k},0)} + \delta_{(-x_{n_k},0)} \right) \right)
\end{displaymath}
is vaguely convergent to some measure $\varphi$.

Define $\psi_k:=\delta_{x_{n_k}}-\delta_{-x_{n_k}}$, which is a measure on $G$. Then, $(\phi_k)_{k\in\N}$ converges vaguely on $G\times \{0\}$ if and only if $(\psi_k)_{k\in\N}$ converges vaguely on $G$, since $G\times\{0\}$ and $G$ are homeomorphic as topological groups. Let $\psi$ denote the vague limit of $(\psi_k)_{k\in\N}$. Now, we repeat the argument of Example~\ref{ex5}, slightly modified.
Since $G$ is compact, there exists a directed set $I$ and a monotone final function $h: I \to \N$, such that the subnet $(x_{n_{h(\beta)}})_{\beta
\in I }$ is convergent to some $a \in G$. It follows immediately that $-a =\lim_\beta -x_{n_{h(\beta)}}$.
Next, since $a =\lim_\beta x_{n_{h(\beta)}}$ and $-a =\lim_\beta -x_{n_{h(\beta)}}$, for all $f \in \Cc(G)$, we have by the continuity of $f$
\[
f(a)=\lim_\beta f(x_{n_{h(\beta)}}) \qquad \text{ and } \qquad 
f(-a)=\lim_\beta f(-x_{n_{h(\beta)}}) \,.
\]

Fix an arbitrary $f \in \Cc(G)$. Now, since $(\psi_n)_{n\in\N}$ converges vaguely to $\psi$, we have $\psi_n(f) \to \psi(f)$, and hence, any subnet of $(\psi_n(f))_{n\in\N}$ converges to $\psi(f)$. In particular,
\begin{displaymath}
\lim_\beta \psi_{h(\beta)}(f)=\psi(f) \,.
\end{displaymath}
Thus, we have
\begin{displaymath}
\psi(f)=\lim_\beta \psi_{h(\beta)}(f)=\lim_\beta f(x_{n_{h(\beta)}})+ f(-x_{n_{h(\beta)}})=f(a)+f(-a) \,.
\end{displaymath}
Therefore, $\psi=\delta_a+\delta_{-a}$ and hence
\begin{displaymath}
\delta_a+\delta_{-a}= \lim_{k\to\infty} \delta_{x_{n_k}}+\delta_{-x_{n_k}} \,.
\end{displaymath}

Next, since $(x_{n_k})_{k\in\N}$ does not converges to $a$, there exists an open set $U\ni a$ and a subsequence $(x_{n_{k_{\ell}}})_{\ell\in\N}$ of $(x_{n_k})_{k\in\N}$ such that, for all $n\in \N$, we have $x_{n_{k_{\ell}}} \notin U$. Since $(x_{n_{k_{\ell}}})_{\ell\in\N}$ does not converges to $-a$, there exists an open set $ V\ni -a$ and a subsequence $(y_{n_m})_{m\in\N}$ of $(x_{n_{k_{\ell}}})_{\ell\in\N}$ (with $y_{n_m}:=x_{n_{k_{\ell_m}}}$) such that, for all $n\in \N$, we have $y_{n_m} \notin V$.

As $a \in U \cap (-V)$, by Urysohn's lemma, there exists some non-negative $f \in \Cc(G)$ such that $f(a)=1$ and $f(x)=0$ for all $x \notin U \cap (-V)$. Since $\delta_a+\delta_{-a}= \lim_{k\to\infty} \delta_{x_{n_k}}+\delta_{-x_{n_k}}$ and $(y_{n_m})_{m\in\N}$ is a subsequence of $(x_{n_k})_{k\in\N}$, we have $\delta_a+\delta_{-a}= \lim_{m\to\infty} \delta_{y_{n_m}}+\delta_{-y_{n_m}}$ and hence
\begin{displaymath}
f(a)+f(-a)=\lim_{m\to\infty} f(y_{n_m})+f(-y_{n_m}) \,.
\end{displaymath}
Note that since $y_{n_m} \notin U$ and hence $y_{n_m} \notin U \cap (-V)$, we have $f(y_{n_m})=0$. Similarly, since $y_{n_m} \notin V$ we have $-y_{n_m} \notin -V$ and hence $-y_{n_m} \notin U \cap (-V)$, which yields $f(-y_{n_m})=0$. Therefore,
\begin{displaymath}
f(a)+f(-a)=\lim_{m\to\infty} 0+0 =0  \,.
\end{displaymath}
But this is not possible as, by construction, $f(a)=1$ and $f(-a) \geqslant 0$.
Therefore, we get a contradiction.

Since we obtained a contradiction, our assumption that $(\gamma_n)_{n\in\N}$ has a convergent subsequence is wrong.
\end{proof}

\subsection*{Acknowledgments}
The work was supported by NSERC with grant 03762-2014 (NS) and by the German Research Foundation (DFG) via research grant 415818660 (TS). The authors are grateful for the support.


\begin{thebibliography}{99}
\bibitem{ARMA1}
L.N. Argabright, J. Gil de Lamadrid, \textit  {Fourier Analysis of Unbounded Measures on Locally
  Compact Abelian Groups}, Memoirs of the Amer.  Math. Soc. \textbf{145}, 1974.

\bibitem{BaGa}
M. Baake, F. G\"ahler,
\textit{Pair correlations of aperiodic inflation rules via
renormalisation: Some interesting examples}, Topol. Appl. \textbf{205}, 4--27, 2016.
\texttt{arXiv:1511.00885}.

\bibitem{BGM}
M. Baake, F. G\"ahler, N. Ma\~nibo, \textit{Renormalisation of pair correlation measures for primitive inflation rules and absence of absolutely continuous diffraction},
Commun. Math. Phys. \textbf{370}, 591--635, 2019.
\texttt{arXiv:1511.00885}.



\bibitem{TAO}
M. Baake, U. Grimm, \textit{Aperiodic Order. Vol.~1: A Mathematical Invitation}, Cambridge University Press, Cambridge, 2013.

\bibitem{TAO2}
M. Baake, U. Grimm (eds.), \textit{Aperiodic Order. Vol.~2: Crystallography and Almost Periodicity}, Cambridge University Press, Cambridge, 2017.

\bibitem{BG}
M.~ Baake, U.~ Grimm, \textit{Renormalisation of pair correlations and their Fourier transforms for primitive block substitutions}, preprint 2019.
\texttt{arXiv:1906.10484}

\bibitem{BG2}
M.~ Baake, U.~ Grimm, \textit{Fourier transform of Rauzy fractals and point spectrum of 1D Pisot inflation tilings}, 2019.
\texttt{arXiv:1907.11012}

\bibitem{BG3}
M.~ Baake, U.~ Grimm, \textit{Diffraction of a model set with complex windows}, 2019.
\texttt{arXiv:1904.08285}

\bibitem{BL}
M. Baake, D. Lenz, \textit{Dynamical systems on translation bounded
measures:\ Pure point dynamical and diffraction spectra}, Ergod.\ Th.\ \& Dynam.\ Syst. \textbf{24}, 1867--1893, 2004.
\texttt{arXiv:math.DS/0302231}


\bibitem{BL2}
M.~Baake, D.~Lenz, \textit{Spectral notions of aperiodic order},
Discrete Cont.~Dynamical~Systems~Ser. S \textbf{10} (2017) 161--190.
\texttt{arXiv:1601.06629}

\bibitem{BM}
M.~Baake, R.V.~Moody, \textit{Weighted Dirac combs with pure point diffraction},
J.\ Reine Angew.\ Math.\ (Crelle) {\bf 573} (2004) 61--94.
\texttt{math.MG/0203030}

\bibitem{BF}
C. Berg, G. Forst, \textit{Potential Theory on Locally Compact Abelian Groups}, Springer, Berlin, 1975.

\bibitem{BHS}
M.~Baake, C.~Huck and N.~Strungaru, \textit{On weak model sets of extremal density}, Indag.\ Math. \textbf{28} 3--31, 2017.
\texttt{arXiv:1512.07129}.

\bibitem{CoHa}
Y. Cornulier, P. Harpe, \textit{Metric geometry of locally compact groups}, EMS Tracts in Mathematics Vol. \textbf{25}, 2016. 
\texttt{arXiv:1403.3796}.

\bibitem{ARMA}
J.~Gil.~de~Lamadrid, L.~N.~Argabright, \textit{Almost Periodic
Measures}, Mem. Amer. Math. Soc., Vol \textbf{85}, No. 428, 1990.

\bibitem{Hof}
A.~Hof,
\textit{On diffraction by aperiodic structures},
Commun.\ Math.\ Phys.\ {\bf 169} 25--43, 1995.

\bibitem{Hof3}
A.~Hof,
\textit{Diffraction by aperiodic structures},
in: \textit{The Mathematics of Long-Range Aperiodic Order}
(ed.\ R.V.\ Moody),
NATO ASI Series C \textbf{489}, Kluwer, Dordrecht, pp. 239--268, 1997.


\bibitem{HR}
E.~ Hewitt, K.A.~Ross, \textit{Abstract Harmonic Analysis} , Springer-Verlag, Berlin, 1963.


\bibitem{KR}
G.~Keller, C.~Richard, \textit{Dynamics on the graph of the torus parametrization}, Ergodic Th.\ \& Dynam.\ Syst.
\textbf{38}, 1048--1085, 1028.
\texttt{arXiv:1511.06137}

\bibitem{LS}
D.~Lenz, N.~Strungaru, \textit{On weakly almost periodic measures}, Trans.\ Amer.\ Math.\ Soc. \textbf{371},
6843--6881, 2019.
\texttt{arXiv:1609.08219}

\bibitem{MoSt}
R.V.~Moody, N.~Strungaru, \textit{Almost Periodic Measures and their Fourier
Transforms}. in \cite{TAO2}, pp. 173--270, 2017.


\bibitem{Osb}
M.S.~Osborne,
\textit{On the Schwartz--Bruhat space and the Paley--Wiener
theorem for locally compact abelian groups}, J.\ Funct.\ Anal. \textbf{19}, 40--49, 1975.

\bibitem{Ped}
G.~K.~Pedersen: \textit{Analysis Now}, Springer, New York, 1989;
Revised \ printing 1995.

\bibitem{CRS}
C. Richard, N. Strungaru, \textit{Pure point diffraction and Poisson Summation}, Ann. H. Poincar\'e \textbf{18}, 3903-–3931, 2017.
\texttt{arXiv:1512.00912}

\bibitem{CRS3}
C. Richard, N. Strungaru, \textit{A short guide to pure point diffraction in cut-and-project sets}, J. Phys. A: Math. Theor. \textbf{50}, no 15, 2017.
\texttt{arXiv:1606.08831}

\bibitem{CRS2}
C. Richard, N. Strungaru, textit{Fourier Analysis of unbounded measures on Lattices in LCAG}, in preparation.

\bibitem{Rud2}
W. Rudin: \textit{ Functional Analysis}, McGraw-Hill, New York, 1991.

\bibitem{Martin2}
M.~Schlottmann, \textit{Generalized model sets and dynamical
systems}, in: \textit{Directions in
Mathematical Quasicrystals}, eds. M. Baake, R.V. Moody, CRM Monogr. Ser., Amer. Math. Soc., Providence,
RI, pp. 143--159, 2000.

\bibitem{SeeSte}
L. A. Steen, J. A. Seebach, \textit{Counterexamples in Topology}, Dover, New York, 1970.


\bibitem{SpSt}
T. Spindeler, N. Strungaru, \textit{On norm almost periodic measures}, preprint, 2019.
\textit{arXiv:1810.09490}

\bibitem{NS11}
N.~Strungaru, \textit{Almost Periodic Pure Point Measures}, in:  \cite{TAO2}, pp. 271--342, 2017.
\texttt{arXiv:1501.00945}.

\bibitem{NS12}
N.~Strungaru, \textit{On the Fourier Transformability of Strongly Almost Periodic Measures}, preprint, to appear in Canad. J. of Math., 2017.
\texttt{arXiv:1704.04778}

\bibitem{NS13}
N.~Strungaru, \textit{On the Fourier Analysis of Measures with Meyer Set Support}, preprint, to appear in J.\ Funct.\ Anal., 2018.
\texttt{arXiv:1807.03815}


\bibitem{ST}
N.~Strungaru,  V.~Terauds, \textit{Diffraction theory and almost
periodic distributions}, J. Stat. Phys. \textbf{164},  1183--1216,
2016. 
\texttt{arXiv:1603.04796}

\end{thebibliography}
\end{document}